\documentclass[11pt]{article}

\usepackage{amssymb,amsmath,mathrsfs,amsthm,latexsym}
\usepackage[a4paper,includeheadfoot,left=15mm,right=15mm,top=15mm,bottom=25mm]{geometry} 
\usepackage{verbatim}
\usepackage{color}
\usepackage{graphicx}
\usepackage{subcaption}

\newtheorem{thm}{Theorem}[section]
\newtheorem{lem}[thm]{Lemma}
\newtheorem{prop}[thm]{Proposition}

\theoremstyle{definition}

\theoremstyle{remark}
\newtheorem{rmk}[thm]{Remark}

\newcommand{\hsforall}{\hspace{1mm}\forall\hspace{1mm}}						 
\renewcommand{\Re}{\operatorname*{Re}}                             
\renewcommand{\Im}{\operatorname*{Im}}                             
\newcommand{\D}{\ensuremath{\,\mathrm{d}}}							         
\newcommand{\Mspacer}{\hspace{0.5mm}}                              
\newcommand{\M}[3]{#1_{#2\Mspacer#3}}                              
\newcommand{\Msup}[4]{#1_{#2\Mspacer#3}^{#4}}                      
\newcommand{\Msups}[5]{#1_{#2\Mspacer#3}^{#4\Mspacer#5}}           
\renewcommand{\geq}{\geqslant}                                     
\renewcommand{\leq}{\leqslant}                                     
\newcommand{\BE}{\begin{equation}}                                 
\newcommand{\EE}{\end{equation}}                                   
\newcommand{\be}{\begin{equation}}                                 
\newcommand{\ee}{\end{equation}}                                   
\newcommand{\BES}{\begin{equation*}}                               
\newcommand{\EES}{\end{equation*}}                                 
\newcommand{\BP}{\begin{pmatrix}}                                  
\newcommand{\EP}{\end{pmatrix}}                                    
\newcommand{\N}{\mathbb{N}}                                        
\newcommand{\Z}{\mathbb{Z}}                                        
\newcommand{\R}{\mathbb{R}}                                        
\newcommand{\C}{\mathbb{C}}                                        
\newcommand{\la}{\lambda}      
\newcommand\re{{\rm e}}     
     \newcommand\IMVP{initial multipoint value problem }                            

\makeatletter
	\newcommand{\biggg}{\bBigg@{3}}
	\newcommand{\Biggg}{\bBigg@{4}}
	\newcommand{\bigggg}{\bBigg@{5}}
	\newcommand{\Bigggg}{\bBigg@{6}}
\makeatother

\newcommand{\superscript}[1]{\ensuremath{^{\textrm{#1}}}}

\newcommand{\Thns}[0]{\superscript{th}}
\newcommand{\Th}[0]{\Thns~}
\newcommand{\stns}[0]{\superscript{st}}
\newcommand{\st}[0]{\stns~}

\def\clap#1{\hbox to 0pt{\hss#1\hss}}
\def\mathclap{\mathpalette\mathclapinternal}
\def\mathclapinternal#1#2{\clap{$\mathsurround=0pt#1{#2}$}}

\numberwithin{equation}{section}

\title{Nonlocal and multipoint boundary value problems \\ for linear evolution equations}
\author{B. Pelloni$^1$ and  D. A. Smith$^2$\\
\footnotesize 1. Department of Mathematics, Heriot-Watt University, UK \\
\footnotesize 2. Yale-NUS College, Singapore \\
\footnotesize email\textup{: \texttt{dave.smith@yale-nus.edu.sg}}
}

\begin{document}
\maketitle

\abstract{
	We derive the solution representation for a large class of nonlocal boundary value problems for linear evolution PDEs with constant coefficients in one space variable.
	The prototypical such PDE is the heat equation, for which problems of this form model physical phenomena in chemistry and for which we formulate and prove a full result.
	We also consider the third order case, which is much less studied and has been shown by the authors to have very different structural properties in general.  
	
	The nonlocal conditions  we consider can be reformulated as \emph{multipoint conditions}, and then an explicit representation for the solution of the problem is obtained by an application of the Fokas transform method.
	The analysis is carried out under the assumption that the problem being solved is well posed, i.e.\ that it admits a unique solution.
	For the second order case, we also give criteria that guarantee well-posedness. 
}

\section{Introduction}

In a variety of applications of PDE models, classical boundary conditions imposed at the boundary of the domain are not representative of the particular phenomenon, and it is necessary to consider \emph{nonlocal} boundary conditions.
A particular example of such nonlocal conditions are \emph{multipoint} conditions relating the value of the solution at the boundary points with the values at some interior points.
A simple example of this is given in~\cite{bastys2005}, where motivation from physical applications, particularly in chemistry, can also be found.
Early work on three-point boundary conditions was done in~\cite{gupta1998,ma1997}, though these works focus on the analysis and proving existence of a solution of possibly nonlinear ODEs of second order.
Indeed, most existing results are limited to the second order case, either linear or nonlinear, although some third order results are presented in~\cite{PS2008a}.
Related important developments have focused on solving PDEs, linear and nonlinear, on networks, including linear networks~\cite{SS2015a}.

A wide class of more general nonlocal problems can be shown to be equivalent to multipoint problems, see section~\ref{ssec:Nonlocal} below.
This class includes problems in which one or more boundary conditions is replaced by a nonlocal condition specifying the integral of the solution on a certain subinterval of the spatial domain.
For a heat conduction problem, this may represent conservation of the internal energy on a subinterval of the spatial domain~\cite{CLW1990a}.
For a diffusion problem, the integral may represent the total mass of a certain chemical within a given region, which could be easily measured using a photometer~\cite{CEH1987a}.
A number of other applications for similar second order problems are described in~\cite{Deh2003a, FS1991b}.

\medskip
In this paper, we make use of the Fokas transform (also known in the literature as the unified transform) to give a general solution to multipoint boundary value problems for {\em linear PDEs of arbitrary  order} of the form
\BE
	q_t+a(-i\partial_x)^nq=0,\qquad x\in(0,1),\quad t>0.
	\label{genpde}
\EE
We assume here that $n\geq 2$
(the special case $n=1$ will be considered elsewhere).
The coefficient  $a$ is assumed to  satisfy the restriction (\ref{arestr}) below, which essentially ensures that the Cauchy initial value problem for the PDE is well posed on $t>0$.
The choice of equations with only the term of highest order spatial derivative may seem special.
However, the analysis of these particular PDEs captures the essential features of the solution also for the case of constant coefficient linear PDEs with lower order terms, see for example~\cite{Pel2005a} for a full discussion and justification of this claim.

\medskip
A typical multipoint boundary value problem is the one studied in~\cite{bastys2005}:

\smallskip
{\em Find the function $q(x,t)$, $x\in(0,1)$, $t>0$ such that $q_t=q_{xx}$,  and $q$ satisfies the given (sufficiently smooth) initial condition $ q(x,0)=q_0(x)$ and the additional conditions}
\BE
	q(0,t)=c_0\,q\left(\frac 1 2 ,t\right)+d_0(t)\quad {\rm and}\quad q(1,t)=c_1\,q\left(\frac 1 2 ,t\right)+d_1(t).
	\label{typmbc}
\EE

We give a more comprehensive solution than provided in previous papers to  a more general form of this problem, for a PDE of arbitrary order and  multipoint conditions linking an arbitrary number of interior points $\eta_1,\ldots,\eta_{m-1}\in(0,1)$, with $\eta_0=0$ and $\eta_m=1$.
Our most complete result is theorem~\ref{thm:HeatIMVP} for the heat equation, but the majority of the analysis is carried out in much greater generality.
In particular, we prove the following theorem.

\begin{thm} \label{thm:MainThm}
Suppose $q(x,t)$ is the solution of initial-$(m+1)$ point value problem of order $n$ defined below in~\eqref{eqn:IMVP}.

Then $q(x,t)$ admits the integral representation (see equation~\eqref{eqn:q.implicit})
\begin{multline*}
	2\pi q(x,t) = \int_{-\infty}^{\infty} \re^{i\lambda x-a\lambda^nt}\hat{q}_0(\lambda)\D \lambda
	- \int_{\partial D_R\cap \C^+} \re^{i\lambda x-a\lambda^nt}\sum_{k=0}^{n-1}f_k^{0}(\lambda;T)\D\lambda \\
	- \int_{\partial D_R\cap \C^-} \re^{i\lambda (x-1)-a\lambda^nt}\sum_{k=0}^{n-1}f_k^{m}(\lambda;T)\D\lambda.
\end{multline*}
In the above expression,  $\hat{q}_0$ is defined by~\eqref{hatsdef} and $D_R$ is given by~\eqref{DR}.
The terms involving the functions $f_k^r$ are obtained as the solution of linear system~\eqref{eqn:DtoN.n}.
In the particular case $n=2$ (respectively, $n=3$), the solution of this system is provided by lemma~\ref{lem:DtoN.soln.2} (respectively, lemma~\ref{lem:DtoN.soln.3}).
\end{thm}

The paper is organised as follows.
In section 2, we formulate the general multipoint condition, and we show how a large class of nonlocal conditions can be reformulated as multipoint conditions, so fall within the scope of the present work.
In section 3, we give a concise introduction to the Fokas transform in general, and then in section 4 we apply it to the general multipoint boundary problem formulated in section 2.
By the end of section~4, the first two claims of theorem~\ref{thm:MainThm} are established.
In sections 5 and 6, we study in detail the second and third order case respectively, both in general and for specific examples.
In particular, in section 5 we state and prove theorem~\ref{thm:HeatIMVP}.
The appendices contain proofs of the lemmata that conclude theorem~\ref{thm:MainThm}.

\section{Formulation of the problem}
Let $m,n\in\N$ be independent,
\BE
	0 = \eta_0 < \eta_1 < \eta_2 <\ldots< \eta_m = 1,
\EE
and $$\Msup{b}{k}{j}{r}\in\C\quad \mbox{for } k,j\in\{0,\ldots,n-1\},\;\;r\in\{0,\ldots,m\}.
$$
Consider the \emph{initial-multipoint value problem}
\begin{subequations} \label{eqn:IMVP}
\begin{align} \label{eqn:IMVP:PDE}
	[\partial_t+a(-i\partial_x)^n] q(x,t) &= 0 & (x,t) &\in (0,1)\times(0,T), \\ \label{eqn:IMVP:IC}
	q(x,0) &= q_0(x) & x &\in[0,1], \\ \label{eqn:IMVP:MC}
	\sum_{k=0}^{n-1}\sum_{r=0}^{m}\Msup{b}{k}{j}{r}\partial_x^k q(\eta_r,t) &= g_j(t) & t &\in [0,T], \quad j\in\{0,1,\ldots,n-1\}, 
\end{align}
\end{subequations}
where we assume $g_j \in C^\infty[0,T]$ with $T>0$ a fixed constant, and that the initial datum is compatible with the multipoint data in the sense that
\BE
q_0\in C^n[0,1] \quad{\rm and}\quad  \sum_{k=0}^{n-1}\sum_{r=0}^{m}\Msup{b}{k}{j}{r}\partial_x^k q_0(\eta_r)=g_j(0).
\label{ICcomp}\EE
We always assume that the coefficient $a$ satisfies
\BE
	a\in
	\begin{cases}
		\{e^{i\theta}:\theta\in[0,\pi]\} & \mbox{if } n \mbox{ even,} \\
		\{i,-i\} & \mbox{if } n \mbox{ odd.}
	\end{cases}
	\label{arestr}
\EE

Assuming that a solution $q\in C^n(0,1)$ exists and is unique, we give a representation of this solution by an application of the {\em Fokas transform} approach.
The Fokas transform is an integral transform flexible enough to allow us to derive a general and effective representation of the solution of any such boundary value problem.

The explicit solution representation we derive can be used to justify {\em a posteriori} the existence and uniqueness assumption.

\subsection{Nonlocal boundary conditions} \label{ssec:Nonlocal}
The multipoint boundary condition \eqref{eqn:IMVP:MC} is actually a rather general {\em nonlocal condition}. 
We first illustrate this observation with an example. 

Consider the initial-nonlocal value problem
\begin{subequations} \label{eqn:INVP.Eg}
\begin{align} \label{eqn:INVP.Eg:PDE}
	[\partial_t-\partial_x^2] q(x,t) &= 0 & (x,t) &\in (0,1)\times(0,T), \\ \label{eqn:INVP.Eg:IC}
	q(x,0) &= q_0(x) & x &\in[0,1], \\ \label{eqn:INVP.Eg:NC1}
	\int_0^{\frac{1}{2}}q(x,t)\D x &= 0 & t &\in [0,T], \\ \label{eqn:INVP.Eg:NC2}
	\int_{\frac{1}{2}}^1(1-x)q(x,t)\D x &= 0 & t &\in [0,T],
\end{align}
\end{subequations}
which can be seen as a generalization of a problem studied by Mantzavinos~\cite{FM2013a}.
We claim that problem~\eqref{eqn:INVP.Eg} is equivalent to an \IMVP belonging to class~\eqref{eqn:IMVP}.
To wit, differentiate both nonlocal conditions~\eqref{eqn:INVP.Eg:NC1}--\eqref{eqn:INVP.Eg:NC2} with respect to $t$, and apply~\eqref{eqn:INVP.Eg:PDE}, to obtain
\BE
	\int_0^{\frac{1}{2}}\partial_x^2q(x,t)\D x = 0, \qquad
	\int_{\frac{1}{2}}^1(1-x)\partial_x^2q(x,t)\D x = 0,
\EE
respectively. Integrating (by parts in the latter), one obtains the multipoint conditions
\begin{subequations}
\begin{align}
	&\partial_x q\left(\frac{1}{2},t\right)-\partial_xq(0,t) = 0 & t &\in [0,T], \\
	&\frac{1}{2}\partial_xq\left(\frac{1}{2},t\right)-q(1,t)+q\left(\frac{1}{2},t\right) = 0 & t &\in [0,T].
\end{align}
\end{subequations}

In general, suppose $J\in\{-1,\ldots,n-1\}$, and consider problem~(2.2), but with~(2.2c) replaced by the multipoint and nonlocal conditions
\begin{subequations} \label{eqn:INVP}
\begin{align} \label{eqn:INVP:MC}
	\sum_{k=0}^{n-1}\sum_{r=0}^{m}\Msup{b}{k}{j}{r}\partial_x^k q(\eta_r,t) &= g_j(t) & t &\in [0,T], \quad j\in\{0,1,\ldots,J\}, \\ \label{eqn:INVP:NC}
	\sum_{k=0}^{n-1}\sum_{r=1}^m\int_{\eta_{r-1}}^{\eta_r}\Msup{b}{k}{j}{r}x^kq(x,t)\D x &= g_j(t) & t &\in [0,T], \quad j\in\{J+1,\ldots,n-1\}.
\end{align}
\end{subequations}

\begin{prop} \label{prop:INVP.IMVP.equiv}
	The initial-nonlocal value problem~\eqref{eqn:IMVP:PDE},~\eqref{eqn:IMVP:IC},~\eqref{eqn:INVP},
	with $q_0$ a compatible initial datum in the sense 
	that $q_0 \in C^{n}[0,1]$,
	\begin{subequations}
	\begin{align}
		\hsforall j &\in\{0,\ldots,J\}, & &\sum_{k=0}^{n-1}\sum_{r=0}^{m}\Msup{b}{k}{j}{r}\partial_x^kq_0(\eta_r)=g_j(0), \\ \label{eqn:INVP:Compatibility}
		\hsforall j &\in\{J+1,\ldots,n-1\}, & &\sum_{k=0}^{n-1}\sum_{r=1}^m\int_{\eta_{r-1}}^{\eta_r}\Msup{b}{k}{j}{r}x^kq_0(x)\D x=g_j(0),
	\end{align}
	\end{subequations}
	is equivalent to an \IMVP of the form~\eqref{eqn:IMVP}.
\end{prop}

\begin{proof}
	We show first that each of the nonlocal conditions implies a multipoint condition.
	For each $j\in\{J+1,\ldots,n-1\}$, differentiating nonlocal condition~\eqref{eqn:INVP:NC} with respect to $t$, and applying the PDE~\eqref{eqn:IMVP:PDE} yields
	\BE \label{eqn:INVP.IMVP.equiv.proof.1}
		\sum_{k=0}^{n-1}\sum_{r=1}^m \int_{\eta_{r-1}}^{\eta_r}\Msup{b}{k}{j}{r}x^k\partial_x^nq(x,t)\D x = \frac{i^n}{a}\frac{\D}{\D t}g_j(t).
	\EE
	Integrating by parts $k+1$ times, the left hand side is equal to
	\BES
		\sum_{k=0}^{n-1}\sum_{r=1}^m \sum_{p=0}^k (-1)^p \Msup{b}{k}{j}{r}\left[ \eta_r^{k-p}\partial_x^{n-p-1}q(\eta_r,t) - \eta_{r-1}^{k-p}\partial_x^{n-p-1}q(\eta_{r-1},t) \right] \\
		= \sum_{k=0}^{n-1}\sum_{r=0}^m \sum_{p=0}^k (-1)^p \Msups{\tilde{b}}{k}{j}{r}{p} \partial_x^{n-p-1}q(\eta_r,t),
	\EES
	where
	\BES
		\Msups{\tilde{b}}{k}{j}{r}{p} = \eta_r^{k-p}\left( \Msup{b}{k}{j}{r}-\Msup{b}{k}{j}{r+1} \right),
	\EES
	defining $\Msup{b}{k}{j}{m+1}=0=\Msup{b}{k}{j}{0}$ for notational convenience.
	Finally, we rewrite the operator
	\BES
		\sum_{k=0}^{n-1}\sum_{p=0}^k (-1)^p \Msups{\tilde{b}}{k}{j}{r}{p} \partial_x^{n-p-1} \mbox{ in the form } \sum_{k=0}^{n-1}\Msup{\hat{b}}{k}{j}{r} \partial_x^k
	\EES
	by defining
	\BES
		\Msup{\hat{b}}{k}{j}{r} = \sum_{\ell=n-1-k}^{n-1} (-1)^{n-1-\ell} \eta_r^{k+\ell+1-n}\left( \Msup{b}{k}{j}{r}-\Msup{b}{k}{j}{r+1} \right).
	\EES
	We have shown that each nonlocal condition~\eqref{eqn:INVP:NC} implies a corresponding multipoint condition
	\begin{align*}
		\sum_{k=0}^{n-1}\sum_{r=0}^{m}\Msup{\hat{b}}{k}{j}{r}\partial_x^k q(\eta_r,t) &= \frac{i^n}{a}\frac{\D}{\D t} g_j(t) & t &\in [0,T], \quad j\in\{J+1,\ldots,n-1\}.
	\end{align*}
	Evaluating at $t=0$ implies compatibility of $q_0$ in the sense of equation~\eqref{ICcomp} with $\hat{b}$ replacing $b$.
	
	The converse argument, beginning at
	\BES
		\sum_{k=0}^{n-1}\sum_{r=0}^{m}\Msup{\hat{b}}{k}{j}{r}\partial_x^k q(\eta_r,t) = \frac{i^n}{a} \gamma_j(t),
	\EES
	allows us to rewrite the left side as the left side of equation~\eqref{eqn:INVP.IMVP.equiv.proof.1}.
	Applying the PDE to the integrand, integrating in time from $0$ to $t$ implies
	\BES
		\sum_{k=0}^{n-1}\sum_{r=1}^m\int_{\eta_{r-1}}^{\eta_r}\Msup{b}{k}{j}{r}x^kq(x,t)\D x - \Gamma_j(t) = \sum_{k=0}^{n-1}\sum_{r=1}^m\int_{\eta_{r-1}}^{\eta_r}\Msup{b}{k}{j}{r}x^kq_0(x)\D x - \Gamma_j(0),
	\EES
	for any choice of antiderivative $\Gamma_j$ of $\gamma_j$.
	Selecting the particular $\Gamma_j$ for which the right hand side evaluates to $0$ (equivalently, compatibility condition~\eqref{eqn:INVP:Compatibility} holds) yields nonlocal condition~\eqref{eqn:INVP:NC}.
\end{proof}

It is clear from the proof that, in place of one or more nonlocal conditions of the form~\eqref{eqn:INVP:NC}, one may specify nonlocal conditions of the form
\BE
	\sum_{k=0}^{n-1}\sum_{r=1}^m\int_{\eta_{r-1}}^{\eta_r}\Msup{b}{k}{j}{r}\partial_x^kq(x,t)\D x = g_j(t) \qquad t \in [0,T].
\EE

\section{The Fokas transform}
The Fokas transform is an integral transform method for solving linear and integrable nonlinear PDE with constant coefficients.
Originally motivated by the quest to extend the inverse scattering transform to the case of boundary value problems (see~\cite{Pel2015a}), this method has evolved into a powerful and more general methodology for deriving an effective integral representation for a variety of linear boundary value problems in two variables. 

Around the turn of the century, the Fokas transform method was developed for linear half-line (one point) and finite interval (two point) initial-boundary value problems~\cite{FG1997a,Fok2001a,FP2001a,Pel2004a,Fok2008a,Smi2012a,fokas2014unified}.
An important recent advance is the generalization of the method to interface problems on a variety of domains~\cite{DPS2014a,APSS2015a,SS2015a,DS2015a,DSS2016a}.

The application of this methodology for a linear PDE of the form~\eqref{genpde} always yields an integral representation over a complex contour, and a relation linking all initial and boundary values, called in the literature the {\em global relation}.
The heart of the solution procedure is the exploitation of the global relation to characterise the representation in terms of only the given data---the resulting mapping is called the {\em generalised Dirichlet to Neumann map}. 

Below, we summarise the ingredients of the method in general~\cite{Fok2008a}, and then turn to the class of problems considered in this paper and derive the associated Dirichlet to Neumann map.

\subsection{Formal solution representation via Green's Theorem}

We consider the PDE~\eqref{genpde} for  $(x,t)\in \Omega=(0,1)\times(0,T)$, where
$T$ denotes a fixed positive constant, and $a$ satisfies the constraint (\ref{arestr}).

Let
\BE
A(x,t,\la)=\re^{-i\la x+a\la^n t}q(x,t),\qquad B(x,t,\la)=\re^{-i\la x+a\la^n t}\sum_{k=0}^{n-1}c_k(\la)\partial_x^kq(x,t),
\label{ABdef}
\EE
where the coefficient polynomials $c_k(\la)$ are defined by the identity 
\BES
	\sum_{k=0}^{n-1}c_k(\la)\partial_x^k = ia\left.\frac{\la^n-\ell^n}{\la-\ell}\right|_{\ell=-i\partial_x}.
\EES
The PDE~\eqref{genpde} can be written in the divergence form
\BES
	A_t-B_x=0.
\EES
Using the two-dimensional Green's theorem, we obtain
\BE
	\int_{\partial \tilde\Omega}[A \D x+B \D t]=0,
	\label{grorig}
\EE
where $\partial \tilde\Omega$ denotes the oriented boundary of any simply connected domain $\tilde\Omega\subset\Omega$.
For $\tilde\Omega=(0,1)\times(0,t)$, this equation yields
\BE
	-\int_0^{1}A(x,0,\la) \D x+\int_0^t B(0,s,\la) \D s-\int_0^t B(1,s,\la) \D s+\int_{0}^1 A(x,t,\la) \D x=0.
	\label{gr0}
\EE
Using (\ref{ABdef}), we write this expression as
\begin{multline} \label{gr1}
	\int_0^1\re^{-i\la x}q(x,0) \D x-\re^{a\la^nt}\int_{0}^1 \re^{-i\la x}q(x,t) \D x \\
	= \int_0^t\re^{a\la^ns}\sum_{k=0}^{n-1}c_k(\la)\partial_x^kq(0,s) \D s
		- \int_0^t\re^{-i\la+a\la^ns}\sum_{k=0}^{n-1}c_k(\la)\partial_x^kq(1,s) \D s,
		\qquad  t>0.
\end{multline}

We use the notation
\BE \label{hatsdef}
	\hat q_0(\la)=\int_0^{1}\re^{-i\la x}q(x,0) \D x,\qquad \hat q(t,\la)=\int_{0}^1 \re^{-i\la x}q(x,t) \D x,\qquad 0<t<T.
\EE
We assume that $q(x,0)$ and the boundary values $\partial_x^kq(0,t)$,  $\partial_x^kq(1,t)$ are sufficiently regular functions, and that they are compatible at the corners of $\Omega$.

Inverting the Fourier transform in (\ref{gr1}) for $q(x,t)$, we obtain the implicit representation
\begin{multline} \label{formalrep}
	q(x,t)=\frac 1 {2\pi}\int_{-\infty}^{\infty}\re^{i\la x-a\la^nt}\left[\hat q_0(\la)-\int_0^t\re^{a\la^ns}\sum_{k=0}^{n-1}c_k(\la)\partial_x^kq(0,s)] \D s \right. \\
	\left. {}+ \re^{-i\la}\int_0^t\re^{a\la^ns}\sum_{k=0}^{n-1}c_k(\la)\partial_x^kq(1,s) \D s\right] \D \la,
\end{multline}
valid for $(x,t)\in\Omega$.

Defining $\C^\pm=\{\lambda\in\C:\pm\Im(\lambda)>0\}$, and the domain 
\BE
D_R= \{\lambda\in\C:|\lambda|>R,\;\Re(a\lambda^n)<0\},
	\label{DR}
\EE
we note that 
\begin{itemize}
\item
$\re^{i\la x}$, with  $x\in[0,1]$,  is analytic and bounded for $\la\in \overline{\C^+}$;
\item
$\re^{i\la (x-1)}$, with  $x\in[0,1]$,  is analytic and bounded for $\la\in \overline{\C^-}$;
\item
$\re^{a\lambda^n t}$, with $t>0$,  is analytic and bounded for $\la\in \overline{D_R}$ for any $R>0$.
\end{itemize}
A straightforward application of Cauchy's theorem and Jordan's lemma~\cite{AF1997a} allows us to deform contours and write~\eqref{formalrep} as
\begin{multline} \label{formalrepc1}
	q(x,t)=\frac 1 {2\pi}\left[\int_{-\infty}^{\infty}\re^{i\la x-a\la^nt}\hat q_0(\la) \D \la - \int_{\partial D_R^+}\re^{i\la x-a\la^nt}\int_0^t\re^{a\la^ns}\sum_{k=0}^{n-1}c_k(\la)\partial_x^kq(0,s) \D s \D \la \right. \\
	- \left. \int_{\partial D_R^-}\re^{i\la(x-1)-a\la^nt}\int_0^t\re^{a\la^ns}\sum_{k=0}^{n-1}c_k(\la)\partial_x^kq(1,s) \D s \D \la\right],
\end{multline}
with
\BE
	D_R^\pm = D_R\cap \C^\pm.
	\label{DRspm}
\EE
Finally, for $\eta=0,1$ and $\tau\in[t,T]$, analyticity and boundedness of
\BES
	\int_t^\tau\re^{a\la^ns}\sum_{k=0}^{n-1}c_k(\la)\partial_x^kq(\eta,s) \D s
\EES
for $\lambda\in \overline{D_R}$ permits us to extend the limits of the inner integrals from $(0,t)$ to $(0,\tau)$, obtaining
\begin{multline} \label{formalrepc}
	q(x,t)=\frac 1 {2\pi}\left[\int_{-\infty}^{\infty}\re^{i\la x-a\la^nt}\hat q_0(\la) \D \la - \int_{\partial D_R^+}\re^{i\la x-a\la^nt}\int_0^\tau\re^{a\la^ns}\sum_{k=0}^{n-1}c_k(\la)\partial_x^kq(0,s) \D s \D \la \right. \\
	- \left. \int_{\partial D_R^-}\re^{i\la(x-1)-a\la^nt}\int_0^\tau\re^{a\la^ns}\sum_{k=0}^{n-1}c_k(\la)\partial_x^kq(1,s) \D s \D \la\right].
\end{multline}

\subsection{The global relation}
Equation~\eqref{grorig} can be viewed either as an implicit representation of the solution (as we derived above), or as the starting point for determining the unknown boundary values.
To illustrate the latter, let $\tilde\Omega=(\zeta,\eta)\times (0,\tau)$ where $0\leq \zeta<\eta\leq 1$.
Then write equation (\ref{grorig}) in the form of the following {\em global relation}:
\begin{multline} \label{GRab}
	\re^{-i\lambda\zeta}\int_0^\tau\re^{a\la^ns}\sum_{k=0}^{n-1}c_k(\la)\partial_x^kq(\zeta,s) \D s - \re^{-i\lambda\eta}\int_0^\tau\re^{a\la^ns}\sum_{k=0}^{n-1}c_k(\la)\partial_x^kq(\eta,s) \D s \\
	= \int_\zeta^{\eta}\re^{-i\la x}q(x,0) \D x-\re^{a\la^n\tau}\int_\zeta^{\eta}\re^{-i\la x}q(x,\tau) \D x.
\end{multline}
The particular global relation depends on the specific choice of the domain $\tilde\Omega$, but it is important to stress that we view this as a relation between the various boundary values of the solution.
This point of view is justified a posteriori by the general property that, because of their specific analyticity properties, the terms involving the unknown solution values at time $\tau$ (i.e.\ terms involving $q(x,\tau)$) will not contribute to the final solution representation. 

\begin{rmk} \label{rmk:GR.AlternateDerivation}
	An alternate derivation of the above identities, which more closely follows the classical Fourier transform method for linear evolution equations on the full line, proceeds as follows.
	Restricting $q$ to the spatial interval $(\zeta,\eta)$ and applying the Fourier transform to partial differential equation~\eqref{eqn:IMVP:PDE} yields
	\BES
		\frac{\D}{\D t} \int_\zeta^\eta \re^{-i\lambda x} q(x,t) \D x = \int_\zeta^\eta \re^{-i\lambda x} \partial_{xx} q(x,t) \D x.
	\EES
	Integrating by parts twice on the left hand side produces certain boundary terms and the spatial Fourier transform of the restricted $q$.
	Solving the resulting ODE for the spatial Fourier transform of the restricted $q$ on temporal interval $(0,t)$, we obtain equation~\eqref{grorig} without having to apply Green's theorem.
	The solution representation~\eqref{formalrepc} and global relation~\eqref{GRab} follow as above.
\end{rmk}

\section{The implementation of the Fokas transform method for multipoint value problems}
\subsubsection*{Notation}
For $\lambda\in\C$ and $k\in\{0,\ldots,n-1\}$, we denote a primitive $n$\Th root of unity
\begin{align}
	\alpha &= \re^{2\pi i/n}, \\ \intertext{an exponential function}
	E_r(\lambda) &= \re^{-i\lambda\eta_r}, & r &\in\{0,\ldots,m\}, \\ \intertext{the Fourier transform of the initial datum, restricted to $(\eta_{r-1},\eta_r)$}
	\hat{q}_0^r(\lambda) &= \int_{\eta_{r-1}}^{\eta_r} \re^{-i\lambda x}q_0(x) \D x, & r &\in\{1,\ldots,m\}, \\ \intertext{the Fourier transform of the solution at time $\tau$, restricted to $(\eta_{r-1},\eta_r)$} \label{eqn:q.tau.hat} 
	\hat{q}_\tau^r(\lambda) &= \int_{\eta_{r-1}}^{\eta_r} \re^{-i\lambda x} q(x,\tau) \D x, & r &\in\{1,\ldots,m\}, \\ \intertext{a time transform of the value of $\partial_x^k q$ at $x=\eta_r$}
	f_k^r(\lambda) = f_k^r(\lambda;\tau) &= c_k(\lambda)\int_0^\tau \re^{a\lambda^ns}\partial_x^kq(\eta_r,s)\D s, & r &\in\{0,\ldots,m\},\;\; \tau \in [0,T].
\end{align}
For convenience of notation, we usually suppress the explicit $\tau$-dependence of $f_k^r$.

\subsubsection*{The implicit integral representation of the solution}
By implementing the steps of the Fokas transform method outlined in the previous section, we find that the solution $q(x,t)$ can be represented as 
\BE \label{eqn:q.implicit}
	2\pi q(x,t) = \int_{-\infty}^{\infty} \re^{i\lambda x-a\lambda^nt}\hat{q}_0(\lambda)\D \lambda - \int_{\partial D_R^+} \re^{i\lambda x-a\lambda^nt}\sum_{k=0}^{n-1}f_k^{0}(\lambda;\tau)\D\lambda 
	- \int_{\partial D_R^-} \re^{i\lambda (x-1)-a\lambda^nt}\sum_{k=0}^{n-1}f_k^{m}(\lambda;\tau)\D\lambda,
\EE
for $x\in(0,1)$, $t\in(0,T)$, and this representation holds for any choice of $\tau\in(t,T]$ and any $R\geq0$ (see equation~\eqref{formalrepc}).

This representation depends on the Fourier transform of the initial datum $q_0(x)$ and on transforms of the boundary values (and their derivatives)  at $x=0$ and $x=1$, namely $\partial_x^kq(0,\tau)$ and $\partial_x^kq(1,\tau)$, $k\in\{0,\ldots,n-1\}$, which are not explicitly known.
Hence this is a \emph{implicit} representation of the solution. The main question is how to characterise these unknown functions in terms of the known data of the problem. 

\subsection{The global relation}

Using the above notation, we consider the global relation in each of the rectangles $(x,t)\in\tilde\Omega_r=[\eta_{r-1},\eta_r]\times [0,\tau]$, $r\in\{1,\ldots,m\}$, $\lambda\in\C$.
This yields a set of $m$  \emph{global relations}.
By evaluating each relation at $\lambda,\alpha\lambda,\ldots,\alpha^{n-1}\lambda$, and using the fact that $f_k^r(\alpha\lambda) = \alpha^{n-1-k}f_k^r(\lambda)$, we obtain the following system of $mn$ equations:
\BE \label{eqn:GR.general}
	\sum_{k=0}^{n-1}\alpha^{(n-1-k)p}\left[ E_{r-1}(\alpha^p\lambda)f_k^{r-1}(\lambda) - E_{r}(\alpha^p\lambda)f_k^{r}(\lambda) \right] = \hat{q}_0^r(\alpha^p\lambda) - \re^{a\lambda^n\tau}\hat{q}_\tau^r(\alpha^p\lambda).
\EE
Explicitly, for each $p\in\{0,1,\ldots,n-1\}$ we have the following system of $m$ equations:
\begin{eqnarray*}
&&\sum_{k=0}^{n-1}\alpha^{(n-1-k)p}\left[ f_k^{0}(\lambda) - E_{1}(\alpha^p\lambda)f_k^{1}(\lambda) \right] 
= \hat{q}_0^1(\alpha^p\lambda) - \re^{a\lambda^n\tau}\hat{q}_\tau^1(\alpha^p\lambda)\\
&&\sum_{k=0}^{n-1}\alpha^{(n-1-k)p}\left[ E_{1}(\alpha^p\lambda)f_k^{1}(\lambda) - E_{2}(\alpha^p\lambda)f_k^{2}(\lambda) \right] 
=\hat{q}_0^2(\alpha^p\lambda) - \re^{a\lambda^n\tau}\hat{q}_\tau^2(\alpha^p\lambda)\\
&&   \qquad \vdots
\\
&&\sum_{k=0}^{n-1}\alpha^{(n-1-k)p}\left[ E_{m-1}(\alpha^p\lambda)f_k^{m-1}(\lambda) - E_m(\alpha^p\lambda)f_k^{m}(\lambda) \right] 
= \hat{q}_0^m(\alpha^p\lambda) - \re^{a\lambda^n\tau}\hat{q}_\tau^m(\alpha^p\lambda).
\end{eqnarray*}

For the moment, we ignore the terms involving the functions $\hat{q}_\tau^r$; indeed, as we mentioned already, they will not contribute to the solution representation we eventually derive.
We then  have a system of $mn$ equations for the $(m+1)n$ unknown functions $f_k^{r}$, $k\in\{0,\ldots,n-1\}$, $r\in\{0,\ldots,m\}$.
Using the data of the problem, namely the multipoint conditions~\eqref{eqn:IMVP:MC}, the number of equations increases to  $(m+1)n$, the same as the number of unknowns.  

In the next section, we explicitly formulate this $(m+1)n$ dimensional system. 

\subsection{Formulation of the generalised Dirichlet-to-Neumann map}

Applying the time transform to the multipoint conditions (\ref{eqn:IMVP:MC}), we obtain, for $\tau\in[0,T]$ and $j\in\{0,\ldots,n-1\}$,
\BE \label{eqn:MC.transformed}
	\sum_{k=0}^{n-1}\sum_{r=0}^m \Msup{b}{k}{j}{r} \frac{(-a)}{i^nc_k(\lambda)} f_k^r(\lambda;\tau) =  \frac{-a}{i^n}\int_0^\tau \re^{a\lambda^ns} g_j(s) \D s.
\EE
The coefficient $(-a/i^n)$ has the property that $(-a/i^nc_{n-1}(\lambda))=1$, and is included to simplify sightly some of the expressions below.
Combining these $n$ relations with the set of $nm$ global relations we have a system of $(m+1)n$ equations, involving the $(m+1)n$ unknowns $f_k^r(\lambda;\tau)$, $r\in\{0,\ldots,m\}$, $k\in\{0,\ldots,n-1\}$.

\medskip
Set 
\BE
h_j(\lambda) = h_j(\lambda;\tau) := \frac{a}{-i^n}\int_0^\tau \re^{a\lambda^ns} g_j(s) \D s.
\label{hjknown}\EE
The functions $h_j(\lambda)$ are the transforms of the known data of the problem.
The unknowns are collected in the $(m+1)n$-dimensional vector $F$ given by
$$
F=F(\lambda)=(f_0^0(\lambda),\ldots,f_{n-1}^0(\lambda), f_0^1(\lambda),\ldots,f_{n-1}^1(\lambda),\ldots,f_0^m(\lambda),\ldots,f_{n-1}^m(\lambda)).
$$
In an effort to simplify notation we often suppress the $\lambda$-dependence of $F$ and related objects.
For such vectors,  we use the following notational convention:
\BE
F=\Big(\overbrace{f_0^r(\lambda),\ldots,f_{n-1}^r(\lambda)}^{r=0,1,\ldots,m}\Big)
\label{Fordering}\EE
In terms of these functions, 
equations~\eqref{eqn:GR.general} and~\eqref{eqn:MC.transformed}, can be expressed as the linear system
\begin{subequations}
\begin{multline}
F\mathcal{B}
	= \Big(h_0(\lambda),\ldots,h_{n-1}(\lambda),\overbrace{-\hat{q}_0^r(\lambda),-\hat{q}_0^r(\alpha\lambda),\ldots,-\hat{q}_0^r(\alpha^{n-1}\lambda)}^{r=1,2,\ldots,m}\Big) \\
	 + e^{a\lambda^n\tau}
	 \Big(0,\ldots,0,\overbrace{\hat{q}_\tau^r(\lambda),\hat{q}_\tau^r(\alpha\lambda),\ldots,\hat{q}_\tau^r(\alpha^{n-1}\lambda)}^{r=1,2,\ldots,m}\Big),
\end{multline}
where 
the vectors on the right hand side are also $(m+1)n$-dimensional, and are written by convention following the same ordering as $F$ given in~\eqref{Fordering}.
The $(m+1)n\times (m+1)n$ matrix $\mathcal{B}$ is defined  by
\begin{align}
	\mathcal{B} &=
	\BP
		\mathfrak{b}^0 & -e_0 &    0 & \cdots &   0     &   0 \\
		\mathfrak{b}^1 &  e_1 & -e_1 & \cdots &   0     &   0 \\
		\mathfrak{b}^2 &    0 &  e_2 & \cdots &   0     &   0 \\
		\vdots & \vdots & \vdots & \ddots & \vdots & \vdots \\
		\mathfrak{b}^{m-1}& 0 &    0 & \cdots & e_{m-1} & -e_{m-1} \\
		\mathfrak{b}^m &    0 &    0 & \cdots &   0     &  e_m
	\EP,
 \qquad \mbox{with} \\
	\mathfrak{b}^r &=
	\BP
		\Msup{b}{0}{0}{r}\frac{1}{(i\lambda)^{n-1}} & \Msup{b}{0}{1}{r}\frac{1}{(i\lambda)^{n-1}} & \cdots & \Msup{b}{0}{n-1}{r}\frac{1}{(i\lambda)^{n-1}} \\
		\Msup{b}{1}{0}{r}\frac{1}{(i\lambda)^{n-2}} & \Msup{b}{1}{1}{r}\frac{1}{(i\lambda)^{n-2}} & \cdots & \Msup{b}{1}{n-1}{r}\frac{1}{(i\lambda)^{n-2}} \\
		\vdots & \vdots & \ddots & \vdots \\
		\Msup{b}{n-1}{0}{r} & \Msup{b}{n-1}{1}{r} & \cdots & \Msup{b}{n-1}{n-1}{r} \\
	\EP, \qquad n\times n \mbox{ block;} \\
	e_r &=
	\BP
		E_r(\lambda) & E_r(\alpha\lambda) \alpha^{n-1} & \cdots & E_r(\alpha^{(n-1)}\lambda) \alpha^{(n-1)(n-1)} \\
		E_r(\lambda) & E_r(\alpha\lambda) \alpha^{n-2} & \cdots & E_r(\alpha^{(n-1)}\lambda) \alpha^{(n-1)(n-2)} \\
		\vdots & \vdots & \ddots & \vdots \\
		E_r(\lambda) & E_r(\alpha\lambda) & \cdots & E_r(\alpha^{(n-1)}\lambda)\alpha^{(n-1)} \\
	\EP, \qquad n\times n \mbox{ block.}
\end{align}
\end{subequations}

Exploiting the Vandermonde-like structure of $e_r$, we rewrite the system as
\begin{subequations} \label{eqn:DtoN.n}
\begin{multline}
	\biggg(\overbrace{E_r(\lambda)\sum_{k=0}^{n-1}f_{k}^r(\lambda),E_r(\alpha\lambda)\sum_{k=0}^{n-1}\alpha^{n-1-k}f_{k}^r(\lambda),\ldots,E_r(\alpha^{n-1}\lambda)\sum_{k=0}^{n-1}\alpha^{(n-1)(n-1-k)}f_{k}^r(\lambda)}^{r=0,1,\ldots,m}\biggg)\mathcal{A} \\
	= \Big(h_0(\lambda),\ldots,h_{n-1}(\lambda),\overbrace{-\hat{q}_0^r(\lambda),-\hat{q}_0^r(\alpha\lambda),\ldots,-\hat{q}_0^r(\alpha^{n-1}\lambda)}^{r=1,2,\ldots,m}\Big) \\
	 + e^{a\lambda^n\tau}
	 \Big(0,\ldots,0,\overbrace{\hat{q}_\tau^r(\lambda),\hat{q}_\tau^r(\alpha\lambda),\ldots,\hat{q}_\tau^r(\alpha^{n-1}\lambda)}^{r=1,2,\ldots,m}\Big),
\end{multline}
where
\BE
	\mathcal{A} =
	\BP
		\beta^0     & -I &  0 & \cdots & 0 & 0 \\
		\beta^1     &  I & -I & \cdots & 0 & 0 \\
		\beta^2     &  0 &  I & \cdots & 0 & 0 \\
		\vdots & \vdots & \vdots & \ddots & \vdots & \vdots \\
		\beta^{m-1} &  0 &  0 & \cdots & I & -I \\
		\beta^m     &  0 &  0 & \cdots & 0 & I
	\EP,
\EE
that is $\mathcal{A}$ is an $(m+1)\times(m+1)$ block matrix, with each block being an $n\times n$ matrix.
The block $I$ is the $n\times n$ identity matrix and the block $\beta^r$ is defined by
\BE
	\beta^r =
	\BP
		\frac{1}{n}E_r(-\lambda)\sum_{j=0}^{n-1}\Msup{b}{j}{0}{r}\frac{1}{(i\lambda)^{n-1-j}} &
		\cdots &
				\frac{1}{n}E_r(-\lambda)\sum_{j=0}^{n-1}\Msup{b}{j}{n-1}{r}\frac{1}{(i\lambda)^{n-1-j}} 
		\\
		\frac{1}{n}E_r(-\alpha\lambda)\sum_{j=0}^{n-1}\alpha^{j+1}\Msup{b}{j}{0}{r}\frac{1}{(i\lambda)^{n-1-j}}&
		\cdots &
				\frac{1}{n}E_r(-\alpha\lambda)\sum_{j=0}^{n-1}\alpha^{j+1}\Msup{b}{j}{n-1}{r}\frac{1}{(i\lambda)^{n-1-j}}
		 \\
		\vdots &  & \vdots \\
		\frac{1}{n}E_r(-\alpha^{n-1}\lambda)\sum_{j=0}^{n-1}\alpha^{(n-1)(j+1)}\Msup{b}{j}{0}{r}\frac{1}{(i\lambda)^{n-1-j}}&
		\cdots &
				\frac{1}{n}E_r(-\alpha^{n-1}\lambda)\sum_{j=0}^{n-1}\alpha^{(n-1)(j+1)}\Msup{b}{j}{n-1}{r}\frac{1}{(i\lambda)^{n-1-j}}
		 \\
	\EP.
\EE
\end{subequations}

This system, in addition to being simpler, has the convenient property that the quantities which must be substituted into equation~\eqref{eqn:q.implicit},
\BE
	\sum_{k=0}^{n-1}f_{k}^0(\lambda) \qquad\mbox{and}\qquad E_m(\lambda)\sum_{k=0}^{n-1}f_{k}^m(\lambda) ={\rm e}^{-i\lambda}\sum_{k=0}^{n-1}f_{k}^m(\lambda),
\EE
are precisely the $1$\st and $(mn+1)$\st unknown quantities.
These two unknowns are the only ones for which we need explicit expressions.

\subsection{Explicit expression for the generalised Dirichlet-to-Neumann map}
To obtain the Dirichlet-to-Neumann map, we must solve the system  (\ref{eqn:DtoN.n}).
However, any solution of this linear system, i.e.\ any expression for $f_k^r(\lambda)$ obtained by solving it (assuming  the system is uniquely solvable), must necessarily depend also upon the \emph{unknown} functions $\hat{q}_\tau^r(\lambda)$, which are the Fourier transform of the solution at the time $t=\tau$.

An important feature of the Fokas transform approach is that the contribution of such terms can usually be proved to vanish.
Indeed, for well-posed boundary value problems for the PDE \eqref{genpde}, it can be proved  that any term involving the unknown functions $\hat{q}_\tau^r(\lambda)$ is bounded and analytic inside the specific contour along which the term is integrated.
Therefore these terms \emph{do not contribute to the solution representation}.
Indeed, the condition that the contribution of these terms can be eliminated is precisely the condition characterizing the class of boundary conditions that yield a well posed problem~\cite{Pel2004a}.

\medskip
It is crucial for our purposes that the same property hold in the case of multipoint boundary value problems. Indeed, we need to establish the following results:
\begin{enumerate}
\item[(a)]
Characterise the class of multipoint boundary conditions that yield a solution of the system (\ref{eqn:DtoN.n}) with analyticity properties that imply that the contribution of any term involving the unknown functions $\hat{q}_\tau^r(\lambda)$ is bounded and analytic inside $D_R^\pm$.
\item[(b)]
Solve the system explicitly for the conditions as in part (a). 
\end{enumerate}

We will not give the full characterisation in part~(a), but rather assume that the multipoint conditions we have are \emph{admissible} in the sense of~\cite{FP2001a,Pel2004a}, i.e.\ they yield a well posed problem which admits a unique solution.
We will however present a discussion and some specific criteria for well-posedness. 

\smallskip
Part~(b) is theoretically straightforward, as the solution of the linear system is simply given by an application of Cramer's rule.
However, deriving an explicit formula via Cramer's rule is not straightforward, due to the size and complexity of the matrix $\mathcal{A}$.

If $m=1$, then the Dirichlet-to-Neumann map~\eqref{eqn:DtoN.n} is $2n\times2n$, but a simpler $n\times n$ formulation has been found by exploiting the adjoint boundary conditions~\cite{FS2016a}, or by directly reducing the $2n\times2n$ system~\cite{Smi2012a}.
It is expected that similar approaches may be applied for $m>1$, but it is not necessary to do so in order to achieve~(b) in some generality.

\smallskip
We must mention here another important issue.
The determinant $\det\mathcal{A}=\Delta(\la)$ is in general an exponential polynomial function of the complex parameter $\la$, therefore it will have countably many zeros $\la_j\in\C$.
The location of these zeros depends on the particular multipoint conditions, but can be estimated asymptotically using general results in complex analysis~\cite{Lan1931a}.
In certain cases, for example when the operator is self-adjoint, it is possible to deform the contour integral solution representation~\eqref{eqn:q.implicit} onto small circular contours about these zeros of $\Delta$, and, via a residue calculation, obtain a series representation of the solution to the initial-multipoint value problem.
However, we emphasize that, even for $m=1$, it is known that it is not always possible to obtain such a series representation~\cite{FS2016a, Pel2005a}.
We leave the study of the criteria that guarantee the existence of such alternative series solutions to a subsequent paper.

\smallskip
In what follows we analyse this system for the case of second and third order, i.e.\ $n=2$ or $3$, and derive explicit formulae for the solution.
The second order case appears most commonly in the literature, and has direct applications \cite{bastys2005}.
We include consideration of the third order case as the solution generally has a very different behaviour.
Heuristically, this is due to the effect of  the boundary conditions destroying the self-adjoint structure of the spatial operator, as discussed in~\cite{PS2013a}.
Indeed, for third order problems, the operator may be degenerate irregular (in the sense of~\cite{Loc2000a,Loc2008a}), yet yield well-posed problems.
The spectral theory associated with such problems is strikingly different to that for Birkhoff-regular problems~\cite{FS2016a}.
Such degenerate-irregular well-posed problems do not occur for $n=2$.

In parallel with the situation for two-point initial-boundary value problems, we expect that $n=2$ and $n=3$ are typical of even and odd order multipoint problems, and the higher order cases add technical challenges but no new mathematical properties.

\section{The case $n=2$: PDEs of second order}

In this section, we solve system~\eqref{eqn:DtoN.n} explicitly.
Exploiting the linearity, we can separate the contributions to the solution of the terms $h_j$, $\hat{q}_0^r$, and $\hat{q}_\tau^r$.
This is particularly convenient for the practical purpose of obtaining an effective integral representation from equation~\eqref{formalrep}.
Indeed, we will show that the terms involving $\hat{q}_\tau^r$ do not contribute to the solution representation. 

\subsection{The Dirichlet to Neumann map}

In the case $n=2$, the linear system~\eqref{eqn:DtoN.n} has dimension $2(m+1)$ and may be expressed as
\begin{subequations} \label{eqn:DtoN.2}
\begin{multline} \label{eqn:DtoN.2.system}
	\Big(\overbrace{E_r(\lambda)[f_{0}^r(\lambda)+f_1^r(\lambda)],E_r(-\lambda)[-f_{0}^r(\lambda)+f_1^r(\lambda)]}^{r=0,1,\ldots,m}\Big)\mathcal{A} \\
	= \Big(h_0(\lambda),h_1(\lambda),\overbrace{-\hat{q}_0^r(\lambda),-\hat{q}_0^r(-\lambda)}^{r=1,2,\ldots,m}\Big) + \re^{a\lambda^n\tau} \Big(0,0,\overbrace{\hat{q}_\tau^r(\lambda),\hat{q}_\tau^r(-\lambda)}^{r=1,2,\ldots,m}\Big),
\end{multline}
where
\BE \label{eqn:DtoN.2.A}
	\mathcal{A} =
	\BP
		\Msup{B}{0}{0}{0} & \Msup{B}{0}{1}{0} & -1 &  0 &  0 &  0 & \cdots & 0 & 0 & 0 & 0 \\
		\Msup{B}{1}{0}{0} & \Msup{B}{1}{1}{0} &  0 & -1 &  0 &  0 & \cdots & 0 & 0 & 0 & 0 \\
		\Msup{B}{0}{0}{1} & \Msup{B}{0}{1}{1} &  1 &  0 & -1 &  0 & \cdots & 0 & 0 & 0 & 0 \\
		\Msup{B}{1}{0}{1} & \Msup{B}{1}{1}{1} &  0 &  1 &  0 & -1 & \cdots & 0 & 0 & 0 & 0 \\
		\Msup{B}{0}{0}{2} & \Msup{B}{0}{1}{2} &  0 &  0 &  1 &  0 & \cdots & 0 & 0 & 0 & 0 \\
		\Msup{B}{1}{0}{2} & \Msup{B}{1}{1}{2} &  0 &  0 &  0 &  1 & \cdots & 0 & 0 & 0 & 0 \\
		\vdots & \vdots & \vdots & \vdots & \vdots & \vdots & \ddots & \vdots & \vdots & \vdots & \vdots \\
		\Msup{B}{0}{0}{m-1} & \Msup{B}{0}{1}{m-1} &  0 &  0 &  0 &  0 & \cdots & 1 & 0 & -1 & 0 \\
		\Msup{B}{1}{0}{m-1} & \Msup{B}{1}{1}{m-1} &  0 &  0 &  0 &  0 & \cdots & 0 & 1 & 0 & -1	\\	
		\Msup{B}{0}{0}{m} & \Msup{B}{0}{1}{m} &  0 &  0 &  0 &  0 & \cdots & 0 & 0 & 1 & 0 \\
		\Msup{B}{1}{0}{m} & \Msup{B}{1}{1}{m} &  0 &  0 &  0 &  0 & \cdots & 0 & 0 & 0 & 1		
	\EP,
\EE
and, for $k\in\{0,1\}$, $r\in\{0,\ldots,m\}$,
\BE \label{eqn:DtoN.2.B}
	\Msup{B}{0}{k}{r} (\la)= \frac{1}{2}E_r(-\lambda)\left[ \Msup{b}{0}{k}{r}\frac{1}{i\lambda} + \Msup{b}{1}{k}{r}\right], \qquad
	\Msup{B}{1}{k}{r}(\la) = \frac{1}{2}E_r( \lambda)\left[-\Msup{b}{0}{k}{r}\frac{1}{i\lambda} + \Msup{b}{1}{k}{r}\right]=\Msup{B}{0}{k}{r} (-\la).
\EE
\end{subequations}
So $\mathcal{A}$ is the matrix with $1$ on the diagonal, $-1$ on the second super-diagonal, and $0$ elsewhere, with the first two columns replaced as shown.

\begin{lem} \label{lem:DtoN.soln.2}
	\textup{\textbf{(a)}} The linear system
	\BE \label{eqn:DtoN.solnlem.2:system}
		(x_0,X_0,x_1,X_1,\ldots,x_m,X_m) \mathcal{A} = (0,0,y_1,Y_1,y_2,Y_2,\ldots,y_m,Y_m),
	\EE
	with $\mathcal{A}$ given by equation~\eqref{eqn:DtoN.2.A} has solution
	\begin{multline} \label{eqn:DtoN.solnlem.2:x_r}
		x_r(\lambda) = \frac{1}{\Delta(\lambda)}
			\Bigg[
			\sum_{j,k=0}^{m} \left[\Msup{B}{0}{0}{j}\Msup{B}{1}{1}{k}-\Msup{B}{1}{0}{k}\Msup{B}{0}{1}{j}\right]
			\left\{\begin{smallmatrix}1\text{\textup{ if }}j<r \\ 0\text{\textup{ if }}j=r \\ -1\text{\textup{ if }}j>r \end{smallmatrix}\right\} \sum_{l=\min\{j,r\}+1}^{\max\{j,r\}} y_l(\lambda) \\
			+
			\frac{1}{2}\sum_{j,k=0}^{m} \left[\Msup{B}{1}{0}{j}\Msup{B}{1}{1}{k}-\Msup{B}{1}{0}{k}\Msup{B}{1}{1}{j}\right]
			\left\{\begin{smallmatrix}1\text{\textup{ if }}j<k \\ 0\text{\textup{ if }}j=k \\ -1\text{\textup{ if }}j>k \end{smallmatrix}\right\} \sum_{l=\min\{j,k\}+1}^{\max\{j,k\}} Y_l(\lambda)
			\Bigg],
	\end{multline}
	and
	\begin{multline} \label{eqn:DtoN.solnlem.2:X_r}
		X_r(\lambda) = \frac{1}{\Delta(\lambda)}
			\Bigg[
			\sum_{j,k=0}^{m} \left[\Msup{B}{0}{0}{j}\Msup{B}{1}{1}{k}-\Msup{B}{1}{0}{k}\Msup{B}{0}{1}{j}\right]
			\left\{\begin{smallmatrix}1\text{\textup{ if }}k<r\vphantom{j} \\ 0\text{\textup{ if }}k=r\vphantom{j} \\ -1\text{\textup{ if }}k>r\vphantom{j} \end{smallmatrix}\right\} \sum_{l=\min\{k,r\}+1}^{\max\{k,r\}} Y_l(\lambda) \\
			-
			\frac{1}{2}\sum_{j,k=0}^{m} \left[\Msup{B}{0}{0}{j}\Msup{B}{0}{1}{k}-\Msup{B}{0}{0}{k}\Msup{B}{0}{1}{j}\right]
			\left\{\begin{smallmatrix}1\text{\textup{ if }}j<k \\ 0\text{\textup{ if }}j=k \\ -1\text{\textup{ if }}j>k \end{smallmatrix}\right\} \sum_{l=\min\{j,k\}+1}^{\max\{j,k\}} y_l(\lambda)
			\Bigg],
	\end{multline}
	where
	\BE \label{eqn:DtoN.solnlem.2:Delta}
		\Delta(\lambda) = \sum_{j,k=0}^m \left[\Msup{B}{0}{0}{j}\Msup{B}{1}{1}{k}-\Msup{B}{1}{0}{k}\Msup{B}{0}{1}{j}\right].
	\EE
	
	\smallskip
	\noindent \textup{\textbf{(b)}} The linear system
	\BE \label{eqn:DtoN.solnlem.2h:system}
		(x_0,X_0,x_1,X_1,\ldots,x_m,X_m) \mathcal{A} = (h_0,h_1,0,\ldots,0),
	\EE
	with $\mathcal{A}$ given by equation~\eqref{eqn:DtoN.2.A} has the ($r$-independent) solution
	\begin{align} \label{eqn:DtoN.solnlem.2h:x_r}
		x_r(\lambda) &= \frac{1}{\Delta(\lambda)} \sum_{j=0}^m \left( \Msup{B}{1}{1}{j} h_0 - \Msup{B}{1}{0}{j} h_1 \right), \\ \label{eqn:DtoN.solnlem.2h:X_r}
		X_r(\lambda) &= \frac{1}{\Delta(\lambda)} \sum_{j=0}^m \left( \Msup{B}{0}{0}{j} h_1 - \Msup{B}{0}{1}{j} h_0 \right),
	\end{align}
	where $\Delta(\lambda)$ is given by equation~\eqref{eqn:DtoN.solnlem.2:Delta}.
\end{lem}

We use lemma~\ref{lem:DtoN.soln.2}, whose proof is presented in appendix~\ref{sec:AppA}, to solve linear system~\eqref{eqn:DtoN.2}, and substitute the solution into
\begin{multline} \label{eqn:q.implicit.2}
	2\pi q(x,t) = \int_\R \re^{i\lambda x-a\lambda^2t}\hat{q}_0^r(\lambda)\D \lambda - \int_{\partial D_R^+} \re^{i\lambda x-a\lambda^2t} \left[ f_0^{0}(\lambda)+f_1^{0}(\lambda) \right] \D\lambda \\
	- \int_{\partial D_R^-} \re^{i\lambda x-a\lambda^2t}\re^{-i\lambda}\left[ f_0^{m}(\lambda)+f_1^{m}(\lambda) \right] \D\lambda,
\end{multline}
which is equation~\eqref{eqn:q.implicit} for the particular value $n=2$.
Explicitly, we find that the relevant data correspond to $x_0$ and $x_m$ and (for the homogenous system, $g_0=g_1=0$) are given by
\begin{subequations} \label{eqn:DtoNMapSolution.n2}
\begin{multline}
	f_{0}^0(\lambda) + f_1^0(\la) = \frac{1}{\Delta(\lambda)}
			\Bigg[\sum_{j=1}^m
			\sum_{k=0}^{m} \left[\Msup{B}{0}{0}{j}\Msup{B}{1}{1}{k}-\Msup{B}{1}{0}{k}\Msup{B}{0}{1}{j}\right]\sum_{l=1}^{j} \hat{q}_0^l(\lambda) \\
			-
			\frac{1}{2}\sum_{j,k=0}^{m} \left[\Msup{B}{1}{0}{j}\Msup{B}{1}{1}{k}-\Msup{B}{1}{0}{k}\Msup{B}{1}{1}{j}\right]
			\left\{\begin{smallmatrix}1\text{\textup{ if }}j<k \\ 0\text{\textup{ if }}j=k \\ -1\text{\textup{ if }}j>k \end{smallmatrix}\right\} \sum_{l=\min\{j,k\}+1}^{\max\{j,k\}} \hat{q}_0^l(-\lambda)\Bigg] + Z^+(\lambda),
\end{multline}
and
\begin{multline}
	\re^{-i\lambda}\left[f_{0}^m(\lambda) + f_1^m(\la)\right] = \frac{-1}{\Delta(\lambda)}
			\Bigg[\sum_{j=0}^{m-1}
			\sum_{k=0}^{m} \left[\Msup{B}{0}{0}{j}\Msup{B}{1}{1}{k}-\Msup{B}{1}{0}{k}\Msup{B}{0}{1}{j}\right]\sum_{l=j+1}^{m} \hat{q}_0^l(\lambda) \\
			+
			\frac{1}{2}\sum_{j,k=0}^{m} \left[\Msup{B}{1}{0}{j}\Msup{B}{1}{1}{k}-\Msup{B}{1}{0}{k}\Msup{B}{1}{1}{j}\right]
			\left\{\begin{smallmatrix}1\text{\textup{ if }}j<k \\ 0\text{\textup{ if }}j=k \\ -1\text{\textup{ if }}j>k \end{smallmatrix}\right\} \sum_{l=\min\{j,k\}+1}^{\max\{j,k\}} \hat{q}_0^l(-\lambda)\Bigg] + Z^-(\lambda).
\end{multline}
\end{subequations}
In this expression, $Z^\pm$ represents terms with $\hat q_0^r$ replaced by $-\hat{q}_\tau^r$.
However, as we show in the next section, these terms are analytic and have sufficient decay inside $D_R^\pm$ to guarantee, using Jordan's lemma, that they do not contribute to the integral representation of the solution.
It follows that, when the corresponding integral from equation~\eqref{eqn:q.implicit} is applied to both sides of one of equations~\eqref{eqn:DtoNMapSolution.n2}, the term involving $Z^\pm$ may be dropped and the equality remains true.

\subsection{The role of analyticity and an effective solution representation}
The solution given above 
does not provide an effective representation of the solution.
Indeed, we have ignored the terms involving the Fourier transform of the solution at time $\tau$, denoted by $\hat{q}_\tau^r$.
In this section, we show that the contribution of the terms involving  $\hat{q}_\tau^r$ vanishes from the solution representation.

The following lemma, whose proof is presented in appendix~\ref{sec:AppB}, provides the essential asymptotic result upon which we rely.

\begin{lem} \label{lem:uniq.solve.lem.2}
	For $\theta\in(0,\frac{\pi}{2})$, define
	\BE
		\C^\pm_\theta := \{ \lambda\in\C : \theta\leq\arg(\pm\lambda)\leq\pi-\theta \}.
	\EE
	\textup{\textbf{(a)}} Suppose the multipoint conditions are such that $\delta^+(\lambda)$ is not identically zero, where
	\BE
		\delta^+(\lambda) := \Msup{B}{0}{0}{0}\Msup{B}{1}{1}{m} - \Msup{B}{1}{0}{m}\Msup{B}{0}{1}{0} = \frac{1}{4}E_m(\lambda)\det
			\BP
				\Msup{b}{0}{0}{0}\frac{1}{i\lambda}+\Msup{b}{1}{0}{0} & \Msup{b}{0}{1}{0}\frac{1}{i\lambda}+\Msup{b}{1}{1}{0} \\
				-\Msup{b}{0}{0}{m}\frac{1}{i\lambda}+\Msup{b}{1}{0}{m} & -\Msup{b}{0}{1}{m}\frac{1}{i\lambda}+\Msup{b}{1}{1}{m}
			\EP.
	\EE
	Then, for all $\theta$, as $\lambda\to\infty$ from within $\C^+_\theta$ away from zeros of $\Delta$, the solution of system~\eqref{eqn:DtoN.solnlem.2:system}, with $y_r=\hat{q}_\tau^r(\lambda)$ and $Y_r=\hat{q}_\tau^r(-\lambda)$ is
	\BE
		x_0(\lambda) = \frac{\gamma^+(\lambda)}{\delta^+(\lambda)} + \mathcal{O}(\lambda^{-1}),
	\EE
	where
	\BE
		\gamma^+(\lambda) = \frac{-1}{2\lambda^2}E_m(\lambda) \det \BP g_0(\tau) & g_1(\tau) \\ \Msup{b}{1}{0}{m} & \Msup{b}{1}{1}{m} \EP.
	\EE
	\textup{\textbf{(b)}} Suppose the multipoint conditions are such that $\delta^-(\lambda)$ is not identically zero, where
	\BE
		\delta^-(\lambda) := \Msup{B}{0}{0}{m}\Msup{B}{1}{1}{0} - \Msup{B}{1}{0}{0}\Msup{B}{0}{1}{m} = \frac{1}{4}E_m(-\lambda)\det
			\BP
				\Msup{b}{0}{0}{m}\frac{1}{i\lambda}+\Msup{b}{1}{0}{0} & \Msup{b}{0}{1}{m}\frac{1}{i\lambda}+\Msup{b}{1}{1}{0} \\
				-\Msup{b}{0}{0}{0}\frac{1}{i\lambda}+\Msup{b}{1}{0}{m} & -\Msup{b}{0}{1}{0}\frac{1}{i\lambda}+\Msup{b}{1}{1}{m}
			\EP
		=-\delta^+(-\lambda).
	\EE
	Then, for all $\theta$, as $\lambda\to\infty$ from within $\C^-_\theta$ away from zeros of $\Delta$, the solution of system~\eqref{eqn:DtoN.solnlem.2:system}, with $y_r=\hat{q}_\tau^r(\lambda)$ and $Y_r=\hat{q}_\tau^r(-\lambda)$ is
	\BE
		x_m(\lambda) = \frac{\gamma^-(\lambda)}{\delta^-(\lambda)} + \mathcal{O}(E_m(\lambda)\lambda^{-1}),
	\EE
	where
	\BE
		\gamma^-(\lambda) = \frac{-1}{2\lambda^2} \det \BP g_0(\tau) & g_1(\tau) \\ \Msup{b}{1}{0}{0} & \Msup{b}{1}{1}{0} \EP.
	\EE
\end{lem}

It is an immediate corollary of lemma~\ref{lem:uniq.solve.lem.2} that the classical (homogeneous or inhomogeneous) Dirichlet, Neumann, and Robin boundary conditions for $m=1$ have $x_0(\lambda)=\mathcal{O}(\lambda^{-1})$ in $\C^+_\theta$ and $x_m(\lambda)=\mathcal{O}(E_m(\lambda)\lambda^{-1})$ in $\C^-_\theta$. Indeed:
\begin{description}
	\item[1. Dirichlet]{$\Msup{b}{1}{0}{m}=\Msup{b}{1}{1}{m}=\Msup{b}{1}{0}{0}=\Msup{b}{1}{1}{0}=0$, so $\gamma^\pm=0$.}
	\item[2. Neumann]{$\delta^+(\lambda)=\pm E_m(\lambda)/4$, so $\delta^+$ dominates $\gamma^+$. Similarly, $\delta^-(\lambda)=\pm E_m(-\lambda)/4$, so $\delta^-(\lambda)E_m(\lambda)$ dominates $\gamma^-$}
	\item[3. Robin]{$\delta^\pm(\lambda)$ has a nonzero $\mathcal{O}(E_m(\pm\lambda))$ term, so $\delta^+$ dominates $\gamma^+$, and $\delta^-(\lambda)E_m(\lambda)$ dominates $\gamma^-$.}
\end{description}
We now give a few multipoint examples for which the same asymptotic behaviour holds:
\begin{description}
\item[4.]
If the multipoint conditions are all order $0$, then $\Msup{b}{1}{j}{r}=0$ for all $j\in\{0,1,\ldots,n-1\}$ and for all $r\in\{0,1,\ldots,m\}$. Hence $\gamma^\pm=0$, so $x_0(\lambda)$ is $\mathcal{O}(\lambda^{-1})$ and $x_m(\lambda)$ is $\mathcal{O}(E_m(\lambda)\lambda^{-1})$.
Hence, by proposition~\ref{prop:uniq.solve.2} the {\bf Dirichlet initial-multipoint value problem} for the heat equation is uniquely solvable by this method. 
\item[5.]
Similarly, Neumann and Robin initial-multipoint value problems, each defined in the natural way, are uniquely solvable.
\item[6.]
Consider an initial-multipoint value problem~\eqref{eqn:IMVP} with $m\geq1$ and multipoint conditions
\begin{alignat*}{6}
	&q(0,t) +{} &&\sum_{r=1}^{m-1}\left[\Msup{b}{0}{0}{r} + \Msup{b}{1}{0}{r}\partial_x\right] q(\eta_r,t) &&+{} \beta \partial_x &&q(1,t) &&= g_0(t) & \qquad\qquad t &\in [0,T], \\
	&         &&\sum_{r=1}^{m-1}\left[\Msup{b}{0}{1}{r} + \Msup{b}{1}{1}{r}\partial_x\right] q(\eta_r,t) &&+{} &&q(1,t) &&= g_1(t) & t &\in [0,T].
\end{alignat*}
In particular, 
\BES
	\BP
		\Msup{b}{0}{0}{0} & \Msup{b}{0}{1}{0} \\
		\Msup{b}{1}{0}{0} & \Msup{b}{1}{1}{0} \\
		\Msup{b}{0}{0}{m} & \Msup{b}{0}{1}{m} \\
		\Msup{b}{1}{0}{m} & \Msup{b}{1}{1}{m}
	\EP
	=
	\BP
		1 & 0 \\ 0 & 0 \\
		0 & 1 \\ \beta & 0
	\EP,
\EES
so
\BES
	\delta^+(\lambda) = \frac{1}{4}E_m(\lambda)\det\BP\frac{1}{i\lambda}&0\\\beta&\frac{-1}{i\lambda}\EP = \frac{1}{4\lambda^2}E_m(\lambda).
\EES
It is immediate that $\gamma^-(\lambda)=0$.
But we would need to show that $\gamma^+(\lambda)=0$ in order to conclude $x_0(\lambda)=\mathcal{O}(\lambda^{-1})$.
However
\BES
	\gamma^+(\lambda) = \frac{1}{2\lambda^2} E_m(\lambda) \beta g_1(\tau).
\EES
Hence, provided the second multipoint condition is homogeneous (or provided it is at least possible to find $\tau\in[t,T]$ such that $g_1(\tau)=0$), we have that $x_0(\lambda)=\mathcal{O}(\lambda^{-1})$.
\item[7.]
For general inhomogeneous data, it is still possible to show that unique solvability holds, using an adaptation of the ``extension of spatial domain'' argument in~\cite{FP2001a}.
\end{description}

\smallskip
A full classification of the multipoint conditions that yield this asymptotic behaviour is beyond the scope of this paper.
However, we indicate in the next proposition how an $x_0(\lambda)=\mathcal{O}(\lambda^{-1})$, $x_m(\lambda)=\mathcal{O}(E_m(\lambda)\lambda^{-1})$ result from lemma~\ref{lem:uniq.solve.lem.2} implies that the contribution of $\hat{q}_\tau^r$ vanishes from the solution representation.

\begin{prop} \label{prop:uniq.solve.2}
	Suppose that the solution of system~\eqref{eqn:DtoN.solnlem.2:system}, with $y_r=\hat{q}_\tau^r(\lambda)$ and $Y_r=\hat{q}_\tau^r(-\lambda)$ satisfies both
	\begin{alignat}{2}
		x_0(\lambda) &= \mathcal{O}(\lambda^{-1}) & \mbox{ as } \lambda &\to\infty \mbox{ from within } \overline{D^+_R} \\
		x_m(\lambda) &= \mathcal{O}(E_m(\lambda)\lambda^{-1}) & \mbox{ as } \lambda &\to\infty \mbox{ from within } \overline{D^-_R}
	\end{alignat}
	and also $\Re(a)>0$.
	Then, provided $R>0$ is chosen sufficiently large, and for all $\tau\in[t,T]$,
	\begin{align}
		0 &= \int_{\partial D_R^+} \re^{i\lambda x + a\lambda^2(\tau-t)}x_0(\lambda)\D\lambda, \\
		0 &= \int_{\partial D_R^-} \re^{i\lambda x + a\lambda^2(\tau-t)}x_m(\lambda)\D\lambda.
	\end{align}
	\end{prop}

\begin{proof}
	The result follows immediately from Jordan's lemma, provided it can be shown that $x_0$ and $x_m$ are analytic on $\overline{D_R^+}$ and $\overline{D_R^-}$, respectively.
	The functions $x_0$ and $x_m$ are defined as ratios of entire functions, hence they are analytic except at zeros of their (shared) denominator, $\Delta(\lambda)$.
	The function $\lambda^2\Delta(\lambda)$ is an exponential polynomial with pure-imaginary (hence, in particular, collinear) exponents and has exactly the same nonzero zeros as $\Delta(\lambda)$.
	By~\cite{Lan1931a}, the zeros of $\lambda^2\Delta(\lambda)$ lie within a pair of logarithmic strips about the positive and negative real axes.
	Therefore, provided $\Re(a)>0$, it is possible to choose $R>0$ sufficiently large that $\overline{D_R}$ is disjoint from those logarithmic strips.
\end{proof}

As an immediate corollary of proposition~\ref{prop:uniq.solve.2}, we obtain that the {\em terms involving 
the unknown function $\hat{q}_\tau^r$ do not contribute to the solution representation}.
This proves the following theorem.

\begin{thm}[Heat equation] \label{thm:HeatIMVP}
	For $n=2$, $a=1$, multipoint coefficients $\Msup{b}{k}{j}{r}$ that satisfy the criteria of proposition~\ref{prop:uniq.solve.2}, and sufficiently smooth data, applying the method described above to the initial-multipoint value problem~\eqref{eqn:IMVP} yields an effective integral representation of the solution.
	Indeed, for $R$ sufficiently large, the solution may be represented using equation~\eqref{eqn:q.implicit}, in which the values specified in equations~\eqref{eqn:DtoNMapSolution.n2}, with $Z^\pm=0$, are substituted for the sums of spectral functions.
\end{thm}

Proposition~\ref{prop:uniq.solve.2} relies crucially upon two criteria:
\begin{enumerate}
	\item[(i)]{in the 1\st solution of system~\eqref{eqn:DtoN.2}, all terms involving $\hat{q}_\tau^r$ are $\mathcal{O}(\lambda^{-1})$ as $\lambda\to\infty$ from within $\overline{D_R^+}$, and in the $(mn+1)$\st solution of system~\eqref{eqn:DtoN.2}, all terms involving $\hat{q}_\tau^r$ are $\mathcal{O}(E_m(\lambda)\lambda^{-1})$ as $\lambda\to\infty$ from within $\overline{D_R^-}$,}
	\item[(ii)]{there exists $R>0$ such that $\overline{D_R}$ contains no zeros of $\Delta$.}
\end{enumerate}
The principal tool for determining the validity of criterion~(i) is lemma~\ref{lem:uniq.solve.lem.2}.
As noted in the proof of proposition~\ref{prop:uniq.solve.2}, criterion~(ii) holds provided $\Re(a)>0$. We briefly investigate the consequences of failure of either of criteria~(i) or~(ii).

Criterion~(ii) is clearly violated if $a=\pm i$, i.e.\ if the partial differential equation~\eqref{eqn:IMVP:PDE} is the linear Schr\"{o}dinger equation $iq_t+q_{xx}=0$.
It is possible to weaken~(ii) to allow zeros of $\Delta$ on $\R\subset\partial D$, but in order to exploit this weakening, a more careful analysis of the location of zeros of $\Delta$ is necessary.
It is even possible to allow infinitely many zeros of $\Delta$ in the interior of $D_R$, provided they are sufficiently asymptotically close to the boundary.
This precise condition will be described in forthcoming work for two-point problems, and the multipoint equivalent is similar.
However it is not possible to discard condition~(ii) entirely.

In order to decide criterion~(i) for the linear Schr\"{o}dinger equation, it would also be necessary to extend lemma~\ref{lem:uniq.solve.lem.2} to analyse $\lambda\to\infty$ from within $\overline{\C^\pm}$, as part of $\partial D_R$ lies along $\R$ if $a=\pm i$.
Now suppose that criterion~(i) is false. If $\Delta=0$, then $\mathcal{A}$ is not full rank, which means that the multipoint conditions are not linearly independent (by an extension of the argument in~\cite[section~2.2.1.4]{Smi2011a}), and the problem is certainly ill-posed.
Even if $\Delta\neq0$ the Jordan's lemma argument fails in at least one of $\C^\pm$.
So it appears that~(i) is necessary to obtain an effective integral representation, at least via this method.

\subsection{Example: the 3-point problem~\eqref{typmbc}} \label{ssec:egBastys}

Consider the problem, studied by Bastys, Ivanauskas, and Sapagovas, with $m=2$, $\eta_1=\frac{1}{2}$ and multipoint conditions
\BE
	q(0,t)-c_0q(\tfrac{1}{2},t)=h_0(t), \qquad q(1,t)-c_1q(\tfrac{1}{2},t)=h_1(t).
\EE
These conditions correspond to the choice
\BE
	\Msup{b}{0}{0}{0} = 1, \qquad \Msup{b}{0}{0}{1} = -c_0, \qquad \Msup{b}{0}{1}{1} = -c_1, \qquad \Msup{b}{0}{1}{2} = 1, \qquad \Msup{b}{k}{j}{r} = 0 \mbox{ otherwise.}
\EE
In their paper~\cite{bastys2005}, the authors find an explicit expression for the solution using an ad-hoc method based on separation of variables and Green's functions.
For their approach, they require that $\frac{|c_0+c_1|}{2}\neq 1$.

\medskip 
Applying the general steps above, we find that in this case the matrix $\mathcal{A}$ is given by 
\BE
	\mathcal{A} =
	\BP
		\frac{   1}{2i\lambda}E_0(-\lambda) & 0                                   & -1 &  0 &  0 &  0 \\
		\frac{  -1}{2i\lambda}E_0( \lambda) & 0                                   &  0 & -1 &  0 &  0 \\
		\frac{-c_0}{2i\lambda}E_1(-\lambda) & \frac{-c_1}{2i\lambda}E_1(-\lambda) &  1 &  0 & -1 &  0 \\
		\frac{ c_0}{2i\lambda}E_1( \lambda) & \frac{ c_1}{2i\lambda}E_1( \lambda) &  0 &  1 &  0 & -1 \\
		0                                   & \frac{   1}{2i\lambda}E_2(-\lambda) &  0 &  0 &  1 &  0 \\
		0                                   & \frac{  -1}{2i\lambda}E_2( \lambda) &  0 &  0 &  0 &  1
	\EP,
\EE
and the determinant $\Delta$ is
\BE
	\Delta(\lambda) = \frac{i\sin(\lambda/2)}{\lambda^2} \left[ \frac{c_0+c_1}{2}-\cos(\lambda/2) \right],\quad \lambda\neq 0.
\EE
The zeros of $\Delta(\lambda)$ are 
\BE \label{eqn:egBastys.eigval}
	\lambda_n =\frac {(2n+1)\pi}{2},\quad \mu_n=2 \arccos\left(\frac{c_0+c_1}{2}\right)+2n\pi,\quad n\in\Z.
\EE
Using the code at~\cite{Smi2018b}, the solution is plotted in figures~\ref{figs:eg-plot}, for several choices of $(c_0,c_1)$ with homogenous multipoint conditions and consistent initial datum.

\begin{figure}
    \centering
    \begin{subfigure}[b]{0.4\textwidth}
        \includegraphics[width=\textwidth]{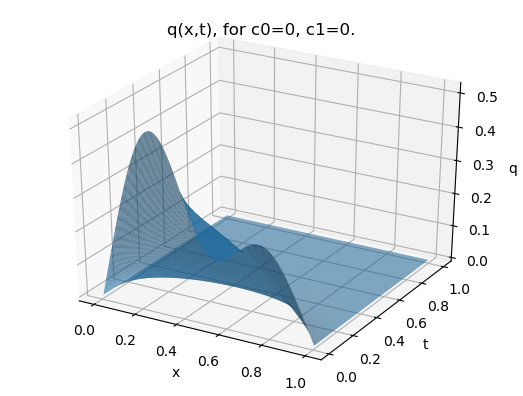}
        \caption{Classical Dirichlet boundary conditions}
        \label{fig:eg-plot-0-0}
    \end{subfigure}
    \hspace{0.1\textwidth}
    \begin{subfigure}[b]{0.4\textwidth}
        \includegraphics[width=\textwidth]{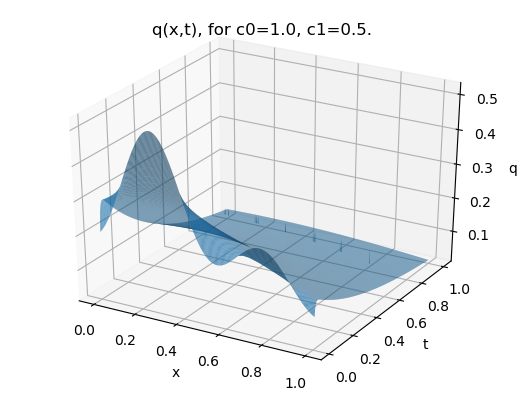}
        \caption{Subcritical multipoint conditions}
        \label{fig:eg-plot-1-0_5}
    \end{subfigure}
	
    \begin{subfigure}[b]{0.4\textwidth}
        \includegraphics[width=\textwidth]{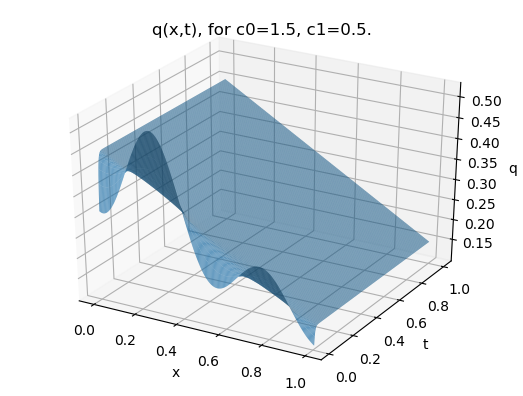}
        \caption{Critical multipoint conditions}
        \label{fig:eg-plot-1_5-0_5}
    \end{subfigure}
    \hspace{0.1\textwidth}
    \begin{subfigure}[b]{0.4\textwidth}
        \includegraphics[width=\textwidth]{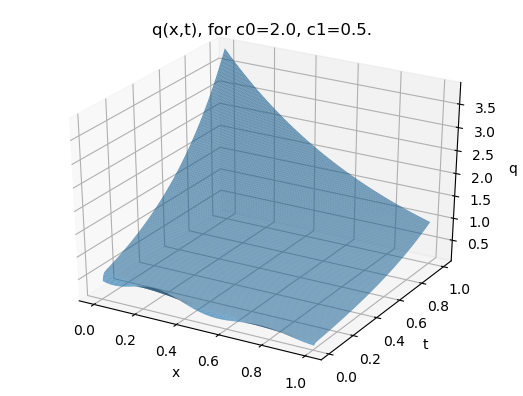}
        \caption{Supercritical multipoint conditions}
        \label{fig:eg-plot-2-0_5}
    \end{subfigure}
	
    \caption{Solution of the 3-point problem~\eqref{typmbc} with initial datum a sum of two box functions, demonstrating the effects of sub/super/critical multipoint conditions on the qualitative behaviour of the solution}\label{figs:eg-plot}
\end{figure}

\begin{rmk} \label{rmk:double.eigval}
	The values such that $\frac{|c_0+c_1|}{2}=1$ correspond to the transition point between real and complex eigenvalues, but also yield the only case in which the eigenvalues are nonsimple.
	The multiplicity of the eigenvalues is the reason that the Green's function approach of~\cite{bastys2005} fails for this case.
	However there is no impediment to using the approach we present.
\end{rmk}

\begin{rmk} \label{rmk:egBastys.SteadyState}
	As suggested by the eigenvalues~\eqref{eqn:egBastys.eigval} and demonstrated by figures~\ref{figs:eg-plot}, the relaxation time of the system increases as the multipoint conditions approach criticality ($c_0+c_1\to2^{-}$), but the steady state solution may be a little surprising for ``Dirichlet-like'' multipoint conditions.
	However, Neumann boundary values may be considered as appropriate limits of Dirichlet-like multipoint conditions which, in view of the below comparison, justifies the phenomenon.
	
	By the usual arguments (see, for example,~\cite[\S2.1.4]{Pin2011a}), the steady state solution must be a linear function $Q(x) = Ax+B$, and it is readily seen that the multipoint conditions imply
	\BES
		A = 2B\left(\frac{c_0-1}{c_0}\right), \qquad A = 2B\left(\frac{c_1-1}{2-c_1}\right).
	\EES
	If $c_0+c_1\neq2$, it follows immediately that $A=B=0$, and the steady state solution is identically zero.
	However, if $c_0+c_1=2$ then this method is insufficient to yield the steady state solution.
	
	This is a manifestation of the effect seen for the classical insulation (homogeneous Neumann) boundary conditions, arising in both cases from the presence of an eigenvalue at $0$.
	For the classical insulation case, after applying the boundary conditions, $A$ is known to be zero, but $B$ is free.
	The usual approach is to either
	\begin{enumerate}
		\item[(a)]{
			appeal to the physics of the problem: by the first law of thermodynamics, the constant steady state temperature must be the mean of the initial temperature profile, or
		}
		\item[(b)]{
			because $v(x,t)=q(x,t)-Q(x)\to0$ as $t\to\infty$ and $v(x,0)=q_0(x)-Q(x)$, and all eigenfunctions but the first decay as $t\to\infty$, it must be that $q_0-Q$ is orthogonal to the first eigenfunction; this yields a new equation for the unknown constant function $Q$.
		}
	\end{enumerate}
	In our method, the contour deformation and residue calculation argument described in~\cite{Smi2011a,Smi2012a,KPPS2017a} may be used to find the contribution of the eigenfunction with eigenvalue at $0$, in terms of the Fourier transform of the initial datum.
	This approach is equivalent to~(b).
	It is not known whether there is a method equivalent to~(a) for this problem; it is not clear that the system should conserve energy.
	
	In the supercritical case $c_0+c_1>2$, the existence of a negative eigenvalue causes the solution to blow up in the long time limit, rather than approach the $Q(x)=0$ steady state solution.
\end{rmk}

\section{The case $n=3$: PDEs of third order}
\subsection{Dirichlet-to-Neumann map}

In the case $n=3$, the linear system~\eqref{eqn:DtoN.n} may be expressed as
\begin{subequations} \label{eqn:DtoN.3}
\begin{multline} \label{eqn:DtoN.3.system}
	\Bigg(\overbrace{E_r(\lambda)\sum_{k=0}^{2}f_{k}^r(\lambda),E_r(\alpha\lambda)\sum_{k=0}^{2}\alpha^{2-k}f_{k}^r(\lambda),E_r(\alpha^{2}\lambda)\sum_{k=0}^{2}\alpha^{2(2-k)}f_{k}^r(\lambda)}^{r=0,1,\ldots,m}\Bigg)\mathcal{A} \\
	= \Big(h_0(\lambda),h_1(\lambda),h_2(\lambda),\overbrace{-\hat{q}_0^r(\lambda),-\hat{q}_0^r(\alpha\lambda),-\hat{q}_0^r(\alpha^2\lambda)}^{r=1,2,\ldots,m}\Big) \\
	+ \re^{a\lambda^n\tau}\hat{q}_\tau^r \Big(0,0,0,\overbrace{\hat{q}_\tau^r(\lambda),\hat{q}_\tau^r(\alpha\lambda),\hat{q}_\tau^r(\alpha^2\lambda)}^{r=1,2,\ldots,m}\Big),
\end{multline}
where
\BE \label{eqn:DtoN.3.A}
	\mathcal{A} =
	\BP
		\Msup{B}{0}{0}{0} & \Msup{B}{0}{1}{0} & \Msup{B}{0}{2}{0} & -1 &  0 &  0 &  0 &  0 &  0 & \cdots & 0 & 0 & 0 \\
		\Msup{B}{1}{0}{0} & \Msup{B}{1}{1}{0} & \Msup{B}{1}{2}{0} &  0 & -1 &  0 &  0 &  0 &  0 & \cdots & 0 & 0 & 0 \\
		\Msup{B}{2}{0}{0} & \Msup{B}{2}{1}{0} & \Msup{B}{2}{2}{0} &  0 &  0 & -1 &  0 &  0 &  0 & \cdots & 0 & 0 & 0 \\
		\Msup{B}{0}{0}{1} & \Msup{B}{0}{1}{1} & \Msup{B}{0}{2}{1} &  1 &  0 &  0 & -1 &  0 &  0 & \cdots & 0 & 0 & 0 \\
		\Msup{B}{1}{0}{1} & \Msup{B}{1}{1}{1} & \Msup{B}{1}{2}{1} &  0 &  1 &  0 &  0 & -1 &  0 & \cdots & 0 & 0 & 0 \\
		\Msup{B}{2}{0}{1} & \Msup{B}{2}{1}{1} & \Msup{B}{2}{2}{1} &  0 &  0 &  1 &  0 &  0 & -1 & \cdots & 0 & 0 & 0 \\
		\Msup{B}{0}{0}{2} & \Msup{B}{0}{1}{2} & \Msup{B}{0}{2}{2} &  0 &  0 &  0 &  1 &  0 &  0 & \cdots & 0 & 0 & 0 \\
		\Msup{B}{1}{0}{2} & \Msup{B}{1}{1}{2} & \Msup{B}{1}{2}{2} &  0 &  0 &  0 &  0 &  1 &  0 & \cdots & 0 & 0 & 0 \\
		\Msup{B}{2}{0}{2} & \Msup{B}{2}{1}{2} & \Msup{B}{2}{2}{2} &  0 &  0 &  0 &  0 &  0 &  1 & \cdots & 0 & 0 & 0 \\
		\vdots & \vdots & \vdots & \vdots & \vdots & \vdots & \vdots & \vdots & \vdots & \ddots & \vdots & \vdots & \vdots \\
		\Msup{B}{0}{0}{m} & \Msup{B}{0}{1}{m} & \Msup{B}{0}{2}{m} &  0 &  0 &  0 &  0 &  0 &  0 & \cdots & 1 & 0 & 0 \\
		\Msup{B}{1}{0}{m} & \Msup{B}{1}{1}{m} & \Msup{B}{1}{2}{m} &  0 &  0 &  0 &  0 &  0 &  0 & \cdots & 0 & 1 & 0 \\
		\Msup{B}{2}{0}{m} & \Msup{B}{2}{1}{m} & \Msup{B}{2}{2}{m} &  0 &  0 &  0 &  0 &  0 &  0 & \cdots & 0 & 0 & 1
	\EP,
\EE
and, for $k\in\{0,1,2\}$, $r\in\{0,\ldots,m\}$,
\begin{align} \label{eqn:DtoN.3.B0}
	\Msup{B}{0}{k}{r} (\la)&= \frac{1}{3}E_r(-\lambda)\left[-\frac{ \Msup{b}{0}{k}{r}}{\lambda^2} +\frac{ \Msup{b}{1}{k}{r}}{i\lambda} + \Msup{b}{2}{k}{r}\right], \\ \label{eqn:DtoN.3.B1}
	\Msup{B}{1}{k}{r}(\la) &= \frac{1}{3}E_r(-\alpha\lambda)\left[-\frac{\alpha \Msup{b}{0}{k}{r}}{\lambda^2} + \frac{\alpha^2\Msup{b}{1}{k}{r}}{i\lambda} + \Msup{b}{2}{k}{r}\right]=\Msup{B}{0}{k}{r} (\alpha\la), \\ \label{eqn:DtoN.3.B2}
	\Msup{B}{2}{k}{r} (\la)&= \frac{1}{3}E_r(-\alpha^2\lambda)\left[ -\frac{\alpha^2\Msup{b}{0}{k}{r}}{\lambda^2} + \frac{\alpha\Msup{b}{1}{k}{r}}{i\lambda} + \Msup{b}{2}{k}{r}\right]=\Msup{B}{0}{k}{r} (\alpha^2\la).
\end{align}
\end{subequations}
So $\mathcal{A}$ is the matrix with $1$ on the diagonal, $-1$ on the third super-diagonal, and $0$ elsewhere, with the first three columns replaced as shown.

\bigskip

In the following lemma, the index $s$ is understood to be an element of the group $\Z_3=\{0,1,2\}$ of integers modulo 3.
The proof is presented in appendix~\ref{sec:AppA}. 

\begin{lem} \label{lem:DtoN.soln.3}
	\textup{\textbf{(a)}} The linear system
	\BE \label{eqn:DtoN.solnlem.3:system}
		\Big(\overbrace{x_0^r,x_1^r,x_2^r}^{r=0,1,\ldots,m}\Big)\mathcal{A} = \Big(0,0,0,\overbrace{y_0^r,y_1^r,y_2^r}^{r=1,2,\ldots,m}\Big),
	\EE
	with $\mathcal{A}$ given by equation~\eqref{eqn:DtoN.3.A} has solution
	\begin{multline} \label{eqn:DtoN.solnlem.3:x_s^r}
		x_s^r = \frac{1}{\Delta(\lambda)}
			\Bigg[
				\sum_{j_0,j_1,j_2=0}^m \mathcal{D}(j_0,j_1,j_2) L(j_s,r,s) \\
				+ \frac{1}{2}\sum_{j_0,j_1,j_2=0}^m \M{\mathcal{D}}{s}{s+1}(j_0,j_1,j_2) L(j_s,j_{s+1},s+1) \\
				+ \frac{1}{2}\sum_{j_0,j_1,j_2=0}^m \M{\mathcal{D}}{s}{s+2}(j_0,j_1,j_2) L(j_s,j_{s+2},s+2)
			\Bigg],
	\end{multline}
	where
	\begin{align} \label{eqn:DtoN.solnlem.3:Delta}
		\Delta(\lambda) &= \sum_{j_0,j_1,j_2=0}^m \mathcal{D}(j_0,j_1,j_2), \\
		 \label{eqn:DtoN.solnlem.3:L}
		 L(j_1,j_2,s) &= \left\{\begin{smallmatrix}1\text{\textup{ if }}j_1<j_2 \\ 0\text{\textup{ if }}j_1=j_2 \\ -1\text{\textup{ if }}j_1>j_2 \end{smallmatrix}\right\} \sum_{k=\min\{j_1,j_2\}+1}^{\max\{j_1,j_2\}} y_s^k(\lambda), \\
		 \label{eqn:DtoN.solnlem.3:D}
		 \mathcal{D}(j_0,j_1,j_2) &= \det
			\BP
				\Msup{B}{0}{0}{j_0} & \Msup{B}{0}{1}{j_0} & \Msup{B}{0}{2}{j_0} \\
				\Msup{B}{1}{0}{j_1} & \Msup{B}{1}{1}{j_1} & \Msup{B}{1}{2}{j_1} \\
				\Msup{B}{2}{0}{j_2} & \Msup{B}{2}{1}{j_2} & \Msup{B}{2}{2}{j_2}
			\EP,
	\end{align}
	and $\M{\mathcal{D}}{s}{p}(j_0,j_1,j_2)$ is the matrix $\mathcal{D}(j_0,j_1,j_2)$ with the row $(\Msup{B}{s}{0}{j_s}, \Msup{B}{s}{1}{j_s}, \Msup{B}{s}{2}{j_s})$ replaced by $(\Msup{B}{p}{0}{j_s}, \Msup{B}{p}{1}{j_s}, \Msup{B}{p}{2}{j_s})$.
	
	\smallskip
	\noindent \textup{\textbf{(b)}} The linear system
	\BE \label{eqn:DtoN.solnlem.3h:system}
		\Big(\overbrace{x_0^r,x_1^r,x_2^r}^{r=0,1,\ldots,m}\Big)\mathcal{A} = \Big(h_0,h_1,h_2,\overbrace{0,0,0}^{\mathclap{r=1,2,\ldots,m}}\Big),
	\EE
	with $\mathcal{A}$ given by equation~\eqref{eqn:DtoN.3.A} has solution
	\BE \label{eqn:DtoN.solnlem.3h:x_s^r}
		x_s^r(\lambda) = \frac{1}{\Delta(\lambda)}\sum_{j_0,j_1=0}^m \left( h_s\Msups{\mathcal{C}}{s+1}{s+2}{s+1}{s+2}(j_0,j_1) + h_{s+1}\Msups{\mathcal{C}}{s+2}{s}{s+1}{s+2}(j_0,j_1) + h_{s+2}\Msups{\mathcal{C}}{s}{s+1}{s+1}{s+2}(j_0,j_1) \right),
	\EE
	where $\Delta(\lambda)$ is given by equation~\eqref{eqn:DtoN.solnlem.3:Delta}, and $\mathcal{C}$ are the boundary coefficient minors
	\BE	\label{eqn:DtoN.solnlem.3.3}
		\Msups{\mathcal{C}}{p}{q}{u}{v}(j_0,j_1) = \det 
		\BP
			\Msup{B}{u}{p}{j_0} & \Msup{B}{u}{q}{j_0} \\
			\Msup{B}{v}{p}{j_1} & \Msup{B}{v}{q}{j_1}		
		\EP.
	\EE
\end{lem}

Rather than  deriving general conditions analogous to the ones found in section 5.2,  that guarantee that the terms involving the unknown function $\hat q^r_\tau(\la)$, we will now compute an explicit example.  

\subsection{An explicit example: a 3-point case for $n=3$}
Consider the initial-multipoint value problem for the PDE $q_t+q_{xxx}=0$ ($n=3$, $a=-i$) obtained by setting $m=2$, $\eta_1=\frac{1}{2}$ and prescribing the three multipoint conditions 
\BE
	q(0,t)-cq(\tfrac{1}{2},t)=0, \qquad q(1,t)=0,\qquad q_x(1,t)=0.
\EE
This example is not motivated by any application of which we are aware, but it is a natural third order analogue of the example considered in section~\ref{ssec:egBastys}.
In addition, these particular conditions are the simplest generalization of the boundary conditions $q(0,t)=q(1,t)=q_x(1,t)=0$, which are known to destroy the self-adjoint structure of the spatial operator.
Hence they are interesting from the purely mathematical point of view, and for this reason we include this as our example.

The given conditions correspond to the choice
\BES
	\BP
		\Msup{b}{0}{0}{0} & \Msup{b}{0}{1}{0} & \Msup{b}{0}{2}{0} \\
		\Msup{b}{1}{0}{0} & \Msup{b}{1}{1}{0} & \Msup{b}{1}{2}{0} \\
		\Msup{b}{2}{0}{0} & \Msup{b}{2}{1}{0} & \Msup{b}{2}{2}{0} \\
		\Msup{b}{0}{0}{1} & \Msup{b}{0}{1}{1} & \Msup{b}{0}{2}{1} \\
		\Msup{b}{1}{0}{1} & \Msup{b}{1}{1}{1} & \Msup{b}{1}{2}{1} \\
		\Msup{b}{2}{0}{1} & \Msup{b}{2}{1}{1} & \Msup{b}{2}{2}{1} \\
		\Msup{b}{0}{0}{2} & \Msup{b}{0}{1}{2} & \Msup{b}{0}{2}{2} \\
		\Msup{b}{1}{0}{2} & \Msup{b}{1}{1}{2} & \Msup{b}{1}{2}{2} \\
		\Msup{b}{2}{0}{2} & \Msup{b}{2}{1}{2} & \Msup{b}{2}{2}{2}
	\EP
	=
	\BP
		1  & 0 & 0 \\ 0 & 0 & 0 \\ 0 & 0 & 0 \\
		-c & 0 & 0 \\ 0 & 0 & 0 \\ 0 & 0 & 0 \\
		0  & 1 & 0 \\ 0 & 0 & 1 \\ 0 & 0 & 0
	\EP,
	\qquad
	\BP
		\Msup{B}{0}{0}{0} & \Msup{B}{0}{1}{0} & \Msup{B}{0}{2}{0} \\
		\Msup{B}{1}{0}{0} & \Msup{B}{1}{1}{0} & \Msup{B}{1}{2}{0} \\
		\Msup{B}{2}{0}{0} & \Msup{B}{2}{1}{0} & \Msup{B}{2}{2}{0} \\
		\Msup{B}{0}{0}{1} & \Msup{B}{0}{1}{1} & \Msup{B}{0}{2}{1} \\
		\Msup{B}{1}{0}{1} & \Msup{B}{1}{1}{1} & \Msup{B}{1}{2}{1} \\
		\Msup{B}{2}{0}{1} & \Msup{B}{2}{1}{1} & \Msup{B}{2}{2}{1} \\
		\Msup{B}{0}{0}{2} & \Msup{B}{0}{1}{2} & \Msup{B}{0}{2}{2} \\
		\Msup{B}{1}{0}{2} & \Msup{B}{1}{1}{2} & \Msup{B}{1}{2}{2} \\
		\Msup{B}{2}{0}{2} & \Msup{B}{2}{1}{2} & \Msup{B}{2}{2}{2}
	\EP
	=
	\BP
		\frac{-1}{3\lambda^2} & 0 & 0 \\
		\frac{-\alpha}{3\lambda^2} & 0 & 0 \\
		\frac{-\alpha^2}{3\lambda^2} & 0 & 0 \\
		\frac{c}{3\lambda^2}\re^{i\lambda/2} & 0 & 0 \\
		\frac{\alpha c}{3\lambda^2}\re^{i\alpha\lambda/2} & 0 & 0 \\
		\frac{\alpha^2 c}{3\lambda^2}\re^{i\alpha^2\lambda/2} & 0 & 0 \\
		0 & \frac{-1}{3\lambda^2}\re^{i\lambda} & \frac{1}{3i\lambda}\re^{i\lambda} \\
		0 & \frac{-\alpha}{3\lambda^2}\re^{i\alpha\lambda} & \frac{\alpha^2}{3i\lambda}\re^{i\alpha\lambda} \\
		0 & \frac{-\alpha^2}{3\lambda^2}\re^{i\alpha^2\lambda} & \frac{\alpha}{3i\lambda}\re^{i\alpha^2\lambda}
	\EP.
\EES
It follows that
\BES
	\Delta(\lambda) = \frac{\alpha^2-\alpha}{27i\lambda^5}\sum_{k=0}^2\alpha^k\left( \re^{-i\alpha^k\lambda}-c\re^{-i\alpha^k\lambda/2} \right),
\EES
and, applying lemma~\ref{lem:DtoN.soln.3}(a) with $y_s^r(\lambda)=\hat{q}_\tau^r(\alpha^s\lambda)$,
\begin{align*}
	x_0^0(\lambda) &=
		\frac{\alpha^2-\alpha}{27i\lambda^5\Delta(\lambda)} \Bigg[ -c\re^{-i\lambda/2}\hat{q}_\tau^1(\lambda) \\
		&\phantom{=a{}} - \left\{ \alpha\left(\re^{-i\alpha\lambda}-c\re^{-i\alpha\lambda/2}\right) + \alpha^2\left(\re^{-i\alpha^2\lambda}-c\re^{-i\alpha^2\lambda/2}\right) \right\}\left( \hat{q}_\tau^1(\lambda)+\hat{q}_\tau^2(\lambda) \right) \\
		&\phantom{=a{}} + \re^{-i\lambda} \left\{ \alpha\left(\hat{q}_\tau^1(\alpha\lambda)+\left(1-c\re^{i\alpha\lambda/2}\right)\hat{q}_\tau^2(\alpha\lambda)\right) + \alpha^2\left(\hat{q}_\tau^1(\alpha^2\lambda)+\left(1-c\re^{i\alpha^2\lambda/2}\right)\hat{q}_\tau^2(\alpha^2\lambda)\right) \right\} \Bigg], \\
	x_0^2(\lambda) &=
		\frac{\alpha^2-\alpha}{27i\lambda^5\Delta(\lambda)} \re^{-i\lambda} \sum_{k=0}^2 \alpha^k \left( \hat{q}_\tau^1(\alpha^k\lambda) + \left(1-c\re^{i\alpha^k\lambda/2}\right)\hat{q}_\tau^2(\alpha^k\lambda) \right).
\end{align*}
From here, a simple asymptotic analysis yields
\begin{alignat*}{3}
	x_0^0(\lambda) &= \mathcal{O}(\lambda^{-1}) & \mbox{ as } \lambda &\to \infty \mbox{ from within } & &\overline{D_R^+}, \\
	x_0^2(\lambda) &= \mathcal{O}(\re^{-i\lambda}\lambda^{-1}) & \mbox{ as } \lambda &\to \infty \mbox{ from within } & &\overline{D_R^-}.
\end{alignat*}
This asymptotic result plays the role of lemma~\ref{lem:uniq.solve.lem.2}, which enables a Jordan's lemma argument of the type found in the proof of proposition~\ref{prop:uniq.solve.2}.
In this case, the theory of zeros of exponential sums~\cite{Lan1931a} immediately yields that, for $R>0$ sufficiently large, $\Delta$ has no zeros in $\overline{D_R}$.
Thus the terms involving $\hat{q}_\tau^r$ make no contribution to the solution representation.

\begin{rmk}
We stress that the above result holds if and only if $a=-i$; if and only if the partial differential equation is
\BES
	q_t+q_{xxx}=0.
\EES
Choosing $a=i$ with the same multipoint conditions specifies an ill-posed problem.
This is in accordance with the results of~\cite{Pel2005a}, in which the above example is studied for $c=0$, and it is shown that the problem is well-posed for $a=-i$ only.
\end{rmk}

\section{Conclusions}
We have derived explicit formulae for the solution of a large class of nonlocal boundary value problems, by formulating them as multipoint value problems and generalising the machinery of the Fokas transform to multipoint value problems.
While we have derived the set-up for linear PDEs of arbitrary order, the formulae have been derived explicitly for PDEs of order $n=2$ and $n=3$.

Even though the general expression is quite cumbersome, the derivation is algorithmic, and the final result fully explicit.
In specific examples, it is possible to simply apply these formulae to find a general expression for the solution.
This expression is always in the form of a uniformly convergent integral that can be evaluated numerically in a very efficient manner by numerical complex integration, taking advantage of the analyticity properties of the solution to select an integration contour of fast decay, as done in~\cite{FF2008a,KPPS2017a}.
 
We have studied also some specific examples of second and third order problems, and given criteria that guarantee that the solution representation depends effectively only on the given data of the problem, at least for $n=2$. 
The full characterisation of multipoint conditions that yield well posed problems, and the analysis of the cases when the integral representation is equivalent to a series representation, are left for a future publication. 
 
\appendix

\section{Appendix: Proof of Lemma~\ref{lem:DtoN.soln.2} and Lemma~\ref{lem:DtoN.soln.3}} \label{sec:AppA}

\subsection{Linear system for $n=2$}

\begin{proof}[Proof of lemma~\ref{lem:DtoN.soln.2}(a)]
A constructive proof may be obtained via Cramer's rule, Laplace's formula and lemma~\ref{lem:detlem}.
However the argument is somewhat complex.
Instead, we present a direct verification of the validity of the solution.
It must be shown that, for each $r\in\{1,\ldots,m\}$,
\begin{align} \label{eqn:DtoN.solnlem.2.proof:1}
	x_r-x_{r-1} &= y_r, \\ \label{eqn:DtoN.solnlem.2.proof:2}
	X_r-X_{r-1} &= Y_r,
\end{align}
and that, for each $s\in\{0,1\}$,
\BE \label{eqn:DtoN.solnlem.2.proof:3}
	\Delta(\lambda) \sum_{r=0}^m \left( \Msup{B}{0}{s}{r}x_r(\lambda) + \Msup{B}{1}{s}{r}X_r(\lambda) \right) = 0,
\EE
with $x_r$, $X_r$ defined by equations~\eqref{eqn:DtoN.solnlem.2:x_r} and~\eqref{eqn:DtoN.solnlem.2:X_r}.

Fixing $r\in\{1,\ldots,m\}$, it is immediate that
\begin{multline} \label{eqn:DtoN.solnlem.2.proof:4}
	x_r-x_{r-1} = \frac{1}{\Delta(\lambda)} \sum_{j,k=0}^m  \left[\Msup{B}{0}{0}{j}\Msup{B}{1}{1}{k}-\Msup{B}{1}{0}{k}\Msup{B}{0}{1}{j}\right] \\
	\Bigg( \left\{\begin{smallmatrix}1\text{\textup{ if }}j<r \\ 0\text{\textup{ if }}j=r \\ -1\text{\textup{ if }}j>r \end{smallmatrix}\right\} \sum_{l=\min\{j,r\}+1}^{\max\{j,r\}} y_l(\lambda)
	- \left\{\begin{smallmatrix}1\text{\textup{ if }}j<r-1 \\ 0\text{\textup{ if }}j=r-1 \\ -1\text{\textup{ if }}j>r-1 \end{smallmatrix}\right\} \sum_{l=\min\{j,r-1\}+1}^{\max\{j,r-1\}} y_l(\lambda)\Bigg).
\end{multline}
Note that the latter term in expression~\eqref{eqn:DtoN.solnlem.2:x_r} is independent of $r$, so the corresponding terms from $x_{r-1}$ and $x_r$ cancel.
A case-by-case ($j<r-1$, $j=r-1$, $j=r$, $j>r$) evaluation yields that the parenthetical quantity in equation~\eqref{eqn:DtoN.solnlem.2.proof:4} evaluates to $y_r(\lambda)$ for all $j$.
A substitution of representation~\eqref{eqn:DtoN.solnlem.2:Delta} for $\Delta(\lambda)$, and cancellation, completes the proof of equation~\eqref{eqn:DtoN.solnlem.2.proof:1}.
The proof of equation~\eqref{eqn:DtoN.solnlem.2.proof:2} is very similar.

With $s=0$, the left hand side of equation~\eqref{eqn:DtoN.solnlem.2.proof:3} expands to
\begin{multline} \label{eqn:DtoN.solnlem.2.proof:5}
	\sum_{r,j,k=0}^m  \Msup{B}{0}{0}{r} \left[\Msup{B}{0}{0}{j}\Msup{B}{1}{1}{k}-\Msup{B}{1}{0}{k}\Msup{B}{0}{1}{j}\right] \left\{\begin{smallmatrix}1\text{\textup{ if }}j<r \\ 0\text{\textup{ if }}j=r \\ -1\text{\textup{ if }}j>r \end{smallmatrix}\right\} \sum_{l=\min\{j,r\}+1}^{\max\{j,r\}} y_l(\lambda) \\
	- \frac{1}{2} \sum_{r,j,k=0}^m  \Msup{B}{1}{0}{r} \left[\Msup{B}{0}{0}{j}\Msup{B}{0}{1}{k}-\Msup{B}{0}{0}{k}\Msup{B}{0}{1}{j}\right] \left\{\begin{smallmatrix}1\text{\textup{ if }}j<k \\ 0\text{\textup{ if }}j=k \\ -1\text{\textup{ if }}j>k \end{smallmatrix}\right\} \sum_{l=\min\{j,k\}+1}^{\max\{j,k\}} y_l(\lambda) \\
	+ \sum_{r,j,k=0}^m  \Msup{B}{1}{0}{r} \left[\Msup{B}{0}{0}{j}\Msup{B}{1}{1}{k}-\Msup{B}{1}{0}{k}\Msup{B}{0}{1}{j}\right] \left\{\begin{smallmatrix}1\text{\textup{ if }}k<r\vphantom{j} \\ 0\text{\textup{ if }}k=r\vphantom{j} \\ -1\text{\textup{ if }}k>r\vphantom{j} \end{smallmatrix}\right\} \sum_{l=\min\{k,r\}+1}^{\max\{k,r\}} Y_l(\lambda) \\
	+ \frac{1}{2} \sum_{r,j,k=0}^m  \Msup{B}{0}{0}{r} \left[\Msup{B}{1}{0}{j}\Msup{B}{1}{1}{k}-\Msup{B}{1}{0}{k}\Msup{B}{1}{1}{j}\right] \left\{\begin{smallmatrix}1\text{\textup{ if }}j<k \\ 0\text{\textup{ if }}j=k \\ -1\text{\textup{ if }}j>k \end{smallmatrix}\right\} \sum_{l=\min\{j,k\}+1}^{\max\{j,k\}} Y_l(\lambda).
\end{multline}
Informed by the fact that $y_l$, $Y_l$ are linearly independent, we aim to show that the third and fourth terms of expression~\eqref{eqn:DtoN.solnlem.2.proof:5} cancel.
Adding and subtracting $\Msup{B}{1}{0}{j}\Msup{B}{1}{0}{k}\Msup{B}{0}{1}{r}$, it holds that
\BE \label{eqn:DtoN.solnlem.2.proof:6}
	\Msup{B}{0}{0}{r}\left[\Msup{B}{1}{0}{j}\Msup{B}{1}{1}{k}-\Msup{B}{1}{0}{k}\Msup{B}{1}{1}{j}\right]
	=
	\Msup{B}{1}{0}{j}\left[\Msup{B}{0}{0}{r}\Msup{B}{1}{1}{k}-\Msup{B}{1}{0}{k}\Msup{B}{0}{1}{r}\right]
	-
	\Msup{B}{1}{0}{k}\left[\Msup{B}{0}{0}{r}\Msup{B}{1}{1}{j}-\Msup{B}{1}{0}{j}\Msup{B}{0}{1}{r}\right].
\EE
Hence the fourth term of expression~\eqref{eqn:DtoN.solnlem.2.proof:5} is equal to
\begin{multline} \label{eqn:DtoN.solnlem.2.proof:7}
	\frac{1}{2} \Bigg(
		\sum_{r,j,k=0}^m  \Msup{B}{1}{0}{j}\left[\Msup{B}{0}{0}{r}\Msup{B}{1}{1}{k}-\Msup{B}{1}{0}{k}\Msup{B}{0}{1}{r}\right] \left\{\begin{smallmatrix}1\text{\textup{ if }}j<k \\ 0\text{\textup{ if }}j=k \\ -1\text{\textup{ if }}j>k \end{smallmatrix}\right\} \sum_{l=\min\{j,k\}+1}^{\max\{j,k\}} Y_l(\lambda) \\
		-\sum_{r,j,k=0}^m  \Msup{B}{1}{0}{k}\left[\Msup{B}{0}{0}{r}\Msup{B}{1}{1}{j}-\Msup{B}{1}{0}{j}\Msup{B}{0}{1}{r}\right] \left\{\begin{smallmatrix}1\text{\textup{ if }}j<k \\ 0\text{\textup{ if }}j=k \\ -1\text{\textup{ if }}j>k \end{smallmatrix}\right\} \sum_{l=\min\{j,k\}+1}^{\max\{j,k\}} Y_l(\lambda)
	\Bigg).
\end{multline}
Applying the permutation $(j,r,k)\mapsto(r,j,k)$ to the indices in the first sum of expression~\eqref{eqn:DtoN.solnlem.2.proof:7}, and the permutation $(k,r,j)\mapsto(r,j,k)$ to the indices in the second sum of expression~\eqref{eqn:DtoN.solnlem.2.proof:7}, we obtain
\begin{multline} \label{eqn:DtoN.solnlem.2.proof:8}
	\frac{1}{2} \Bigg(
		\sum_{r,j,k=0}^m  \Msup{B}{1}{0}{r} \left[\Msup{B}{0}{0}{j}\Msup{B}{1}{1}{k}-\Msup{B}{1}{0}{k}\Msup{B}{0}{1}{j}\right] \left\{\begin{smallmatrix}1\text{\textup{ if }}r<k\vphantom{j} \\ 0\text{\textup{ if }}r=k\vphantom{j} \\ -1\text{\textup{ if }}r>k\vphantom{j} \end{smallmatrix}\right\} \sum_{l=\min\{r,k\}+1}^{\max\{r,k\}} Y_l(\lambda) \\
		-\sum_{r,j,k=0}^m  \Msup{B}{1}{0}{r} \left[\Msup{B}{0}{0}{j}\Msup{B}{1}{1}{k}-\Msup{B}{1}{0}{k}\Msup{B}{0}{1}{j}\right] \left\{\begin{smallmatrix}1\text{\textup{ if }}k<r\vphantom{j} \\ 0\text{\textup{ if }}k=r\vphantom{j} \\ -1\text{\textup{ if }}k>r\vphantom{j} \end{smallmatrix}\right\} \sum_{l=\min\{k,r\}+1}^{\max\{k,r\}} Y_l(\lambda)
	\Bigg),
\end{multline}
which, by
\BES
	\left\{\begin{smallmatrix}1\text{\textup{ if }}r<k\vphantom{j} \\ 0\text{\textup{ if }}r=k\vphantom{j} \\ -1\text{\textup{ if }}r>k\vphantom{j} \end{smallmatrix}\right\} = - \left\{\begin{smallmatrix}1\text{\textup{ if }}k<r\vphantom{j} \\ 0\text{\textup{ if }}k=r\vphantom{j} \\ -1\text{\textup{ if }}k>r\vphantom{j} \end{smallmatrix}\right\},
\EES
is equal to
\BES
	-\sum_{r,j,k=0}^m  \Msup{B}{1}{0}{r} \left[\Msup{B}{0}{0}{j}\Msup{B}{1}{1}{k}-\Msup{B}{1}{0}{k}\Msup{B}{0}{1}{j}\right] \left\{\begin{smallmatrix}1\text{\textup{ if }}k<r\vphantom{j} \\ 0\text{\textup{ if }}k=r\vphantom{j} \\ -1\text{\textup{ if }}k>r\vphantom{j} \end{smallmatrix}\right\} \sum_{l=\min\{k,r\}+1}^{\max\{k,r\}} Y_l(\lambda).
\EES
Hence the third and fourth terms of expression~\eqref{eqn:DtoN.solnlem.2.proof:5} cancel.
Similarly, the first and second terms in~\eqref{eqn:DtoN.solnlem.2.proof:5} cancel.
This completes the proof of equation~\eqref{eqn:DtoN.solnlem.2.proof:3} for $s=0$.
A similar argument yields equation~\eqref{eqn:DtoN.solnlem.2.proof:3} for $s=1$.
\end{proof}

\begin{proof}[Proof of lemma~\ref{lem:DtoN.soln.2}(b)]
The definitions~\eqref{eqn:DtoN.solnlem.2h:x_r},~\eqref{eqn:DtoN.solnlem.2h:X_r} of $x_r$ and $X_r$ are independent of $r$, so they must satisfy $x_r-x_{r-1}=0=X_r-X_{r-1}$.
A simple evaluation establishes that
\BES
	\sum_{r=0}^m \left( \Msup{B}{0}{0}{r}x_r + \Msup{B}{1}{0}{r}X_r \right) = h_0, \qquad \sum_{r=0}^m \left( \Msup{B}{0}{1}{r}x_r + \Msup{B}{1}{1}{r}X_r \right) = h_1.
	\qedhere
\EES
\end{proof}

\subsection{Linear system for $n=3$}

\begin{proof}[Proof of lemma~\ref{lem:DtoN.soln.3}(a)]
We begin by noting some useful identities.
Evaluating the definition~\eqref{eqn:DtoN.solnlem.3:L} of $L$, at each of $j_s<r-1$, $j_s=r-1$, $j_s=r$, and $j_s>r$,
\BE	\label{eqn:DtoN.solnlem.3.1}
	L(j_s,r,s)-L(j_s,r-1,s)=y_s^r.
\EE
It is also immediately apparent from the definition that
\BE	\label{eqn:DtoN.solnlem.3.2}
	L(j_1,j_2,s)=-L(j_2,j_1,s).
\EE
By a row swap, the boundary coefficient minors have the property
\BE	\label{eqn:DtoN.solnlem.3.4}
	\Msups{\mathcal{C}}{p}{q}{u}{v}(j_0,j_1)=-\Msups{\mathcal{C}}{p}{q}{v}{u}(j_1,j_0).
\EE
Expanding in minors, for any $s,p\in\Z_3$,
\begin{subequations} \label{eqn:DtoN.solnlem.3.5}
\begin{align} \label{eqn:DtoN.solnlem.3.5a}
	\mathcal{D}(j_0,j_1,j_2) &= \Msup{B}{s}{p}{j_s} \Msups{\mathcal{C}}{p+1}{p+2}{s+1}{s+2}(j_{s+1},j_{s+2}) + \Msup{B}{s+1}{p}{j_{s+1}} \Msups{\mathcal{C}}{p+1}{p+2}{s+2}{s}(j_{s+2},j_{s}) + \Msup{B}{s+2}{p}{j_{s+2}} \Msups{\mathcal{C}}{p+1}{p+2}{s}{s+1}(j_{s},j_{s+1}), \\ \label{eqn:DtoN.solnlem.3.5b}
	\M{\mathcal{D}}{s}{s+1}(j_0,j_1,j_2) &= \Msup{B}{s}{p}{j_s} \Msups{\mathcal{C}}{p+1}{p+2}{s}{s+2}(j_{s+1},j_{s+2}) + \Msup{B}{s}{p}{j_{s+1}} \Msups{\mathcal{C}}{p+1}{p+2}{s+2}{s}(j_{s+2},j_{s}) + \Msup{B}{s+2}{p}{j_{s+2}} \Msups{\mathcal{C}}{p+1}{p+2}{s}{s}(j_{s},j_{s+1}), \\ \label{eqn:DtoN.solnlem.3.5c}
	\M{\mathcal{D}}{s}{s+2}(j_0,j_1,j_2) &= \Msup{B}{s}{p}{j_s} \Msups{\mathcal{C}}{p+1}{p+2}{s+1}{s}(j_{s+1},j_{s+2}) + \Msup{B}{s+1}{p}{j_{s+1}} \Msups{\mathcal{C}}{p+1}{p+2}{s}{s}(j_{s+2},j_{s}) + \Msup{B}{s}{p}{j_{s+2}} \Msups{\mathcal{C}}{p+1}{p+2}{s}{s+1}(j_{s},j_{s+1}).
\end{align}
\end{subequations}

We must establish both
\BE \label{eqn:DtoN.solnlem.3.6}
	x_s^r-x_s^{r-1} = y_s^r \qquad \hsforall r\in\{1,2,\ldots,m\},\; s\in\Z_3,
\EE
and
\BE \label{eqn:DtoN.solnlem.3.7}
	2\Delta(\lambda)\sum_{r=0}^m\sum_{s=0}^2 \Msup{B}{s}{p}{r} x_s^r = 0 \qquad \hsforall p\in\Z_3.
\EE

Note that, in the definition~\eqref{eqn:DtoN.solnlem.3:x_s^r} of $x_s^r$, the second and third triple-sums are independent of $r$.
Therefore they cancel in the difference $x_s^r-x_s^{r-1}$.
Hence
\BES
	x_s^r-x_s^{r-1} = \frac{1}{\Delta(\lambda)} \sum_{j_0,j_1,j_2=0}^m \mathcal{D}(j_0,j_1,j_2) [L(j_s,r,s)-L(j_s,r-1,s)] = \frac{1}{\Delta(\lambda)} \Delta(\lambda) y_s^r,
\EES
by equation~\eqref{eqn:DtoN.solnlem.3.1}.
Hence~\eqref{eqn:DtoN.solnlem.3.6} holds.

Substituting the definition of $x_s^r$, expanding the sum over $s$, collating terms with like numbers in the final argument of $L$, and re-expressing as a sum over $s$, the left hand side of equation~\eqref{eqn:DtoN.solnlem.3.7} may be expressed as
\begin{multline*}
	\sum_{s=0}^2 \Bigg[
		2\sum_{r,j_0,j_1,j_2=0}^m \Msup{B}{s}{p}{r} \mathcal{D}(j_0,j_1,j_2) L(j_s,r,s) \\
		+ \sum_{r,j_0,j_1,j_2=0}^m \Msup{B}{s+2}{p}{r} \M{\mathcal{D}}{s}{s+2}(j_0,j_1,j_2) L(j_{s+2},j_s,s) \\
		+ \sum_{r,j_0,j_1,j_2=0}^m \Msup{B}{s+1}{p}{r} \M{\mathcal{D}}{s}{s+1}(j_0,j_1,j_2) L(j_{s+1},j_s,s)
	\Bigg]
\end{multline*}
Using equations~\eqref{eqn:DtoN.solnlem.3.5} to expand the determinants $\mathcal{D}$ in their minors, the bracket evaluates to
\begin{multline} \label{eqn:DtoN.solnlem.3.8}
	2\sum_{r,j_0,j_1,j_2=0}^m \Msup{B}{s}{p}{r} \Msup{B}{s}{p}{j_s} \Msups{\mathcal{C}}{p+1}{p+2}{s+1}{s+2}(j_{s+1},j_{s+2}) L(j_s,r,s) \\
	+ 
	2\sum_{r,j_0,j_1,j_2=0}^m \Msup{B}{s}{p}{r} \Msup{B}{s+1}{p}{j_{s+1}} \Msups{\mathcal{C}}{p+1}{p+2}{s+2}{s}(j_{s+2},j_{s}) L(j_s,r,s) \\
	+ 
	2\sum_{r,j_0,j_1,j_2=0}^m \Msup{B}{s}{p}{r} \Msup{B}{s+2}{p}{j_{s+2}} \Msups{\mathcal{C}}{p+1}{p+2}{s}{s+1}(j_{s},j_{s+1}) L(j_s,r,s) \\
	+ 
	\sum_{r,j_0,j_1,j_2=0}^m \Msup{B}{s+2}{p}{r} \Msup{B}{s}{p}{j_s} \Msups{\mathcal{C}}{p+1}{p+2}{s+1}{s}(j_{s+1},j_{s+2}) L(j_{s+2},j_s,s) \\
	+ 
	\sum_{r,j_0,j_1,j_2=0}^m \Msup{B}{s+2}{p}{r} \Msup{B}{s+1}{p}{j_{s+1}} \Msups{\mathcal{C}}{p+1}{p+2}{s}{s}(j_{s+2},j_{s}) L(j_{s+2},j_s,s) \\
	+ 
	\sum_{r,j_0,j_1,j_2=0}^m \Msup{B}{s+2}{p}{r} \Msup{B}{s}{p}{j_{s+2}} \Msups{\mathcal{C}}{p+1}{p+2}{s}{s+1}(j_{s},j_{s+1}) L(j_{s+2},j_s,s) \\
	+ 
	\sum_{r,j_0,j_1,j_2=0}^m \Msup{B}{s+1}{p}{r} \Msup{B}{s}{p}{j_s} \Msups{\mathcal{C}}{p+1}{p+2}{s}{s+2}(j_{s+1},j_{s+2}) L(j_{s+1},j_s,s) \\
	+ 
	\sum_{r,j_0,j_1,j_2=0}^m \Msup{B}{s+1}{p}{r} \Msup{B}{s}{p}{j_{s+1}} \Msups{\mathcal{C}}{p+1}{p+2}{s+2}{s}(j_{s+2},j_{s}) L(j_{s+1},j_s,s) \\
	+ 
	\sum_{r,j_0,j_1,j_2=0}^m \Msup{B}{s+1}{p}{r} \Msup{B}{s+2}{p}{j_{s+2}} \Msups{\mathcal{C}}{p+1}{p+2}{s}{s}(j_{s},j_{s+1}) L(j_{s+1},j_s,s)
\end{multline}

In the first line of expression~\eqref{eqn:DtoN.solnlem.3.8} there are two identical quadruple-sums.
We switch the roles of indices $r\leftrightarrow j_s$ in one of these sums and see that the summand evaluates to
\BES
	\mbox{(summand of other sum in line 1)}\frac{L(r,j_s,s)}{L(j_s,r,s)}.
\EES
Hence, by equation~\eqref{eqn:DtoN.solnlem.3.2}, the two quadruple-sums in the first line cancel.

In line 5, we apply the map $(r,j_s,j_{s+1},j_{s+2}) \mapsto (j_{s+2},j_{s+1},r,j_s)$, and equation~\eqref{eqn:DtoN.solnlem.3.2} to see that the sum on line 5 cancels with the sum on line 9.

In line 4, we apply the map $(r,j_s,j_{s+1},j_{s+2}) \mapsto (j_{s+2},r,j_{s+1},j_s)$, and equation~\eqref{eqn:DtoN.solnlem.3.4}.
In line 6, we apply the map $(r,j_s,j_{s+1},j_{s+2}) \mapsto (j_{s+2},j_s,j_{s+1},r)$, and equation~\eqref{eqn:DtoN.solnlem.3.2}. Hence line 3 cancels with lines 4 and 6.

In line 7, we apply the map $(r,j_s,j_{s+1},j_{s+2}) \mapsto (j_{s+1},r,j_s,j_{s+2})$, and equation~\eqref{eqn:DtoN.solnlem.3.4}.
In line 8, we apply the map $(r,j_s,j_{s+1},j_{s+2}) \mapsto (j_{s+1},j_s,r,j_{s+2})$, and equation~\eqref{eqn:DtoN.solnlem.3.2}.
Hence line 2 cancels with lines 7 and 8.

We have established that expression~\eqref{eqn:DtoN.solnlem.3.8} evaluates to $0$, so~\eqref{eqn:DtoN.solnlem.3.7} holds.
\end{proof}

\begin{proof}[Proof of lemma~\ref{lem:DtoN.soln.3}(b)]
By definition~\eqref{eqn:DtoN.solnlem.3h:x_s^r}, $x_s^r$ is independent of $r$, so $x_s^r-x_s^{r-1}=0$.
For each $p\in\Z_3$, we expand
\begin{multline*}
	\sum_{r=0}^m \sum_{s=0}^2 \Msup{B}{s}{p}{r} x_s^r = \\
		\frac{1}{\Delta}\sum_{r,j_0,j_1=0}^m \Bigg[ \sum_{s=0}^2 \Msup{B}{s}{p}{r} h_s\Msups{\mathcal{C}}{s+1}{s+2}{s+1}{s+2}(j_0,j_1) + \sum_{s=0}^2 \Msup{B}{s}{p}{r} h_{s+1}\Msups{\mathcal{C}}{s+2}{s}{s+1}{s+2}(j_0,j_1) + \sum_{s=0}^2 \Msup{B}{s}{p}{r} h_{s+2}\Msups{\mathcal{C}}{s}{s+1}{s+1}{s+2}(j_0,j_1) \Bigg].
\end{multline*}
As $\{s,s+1,s+2\}=\{0,1,2\}=\{p,p+1,p+2\}$ are all representations of $\Z_3$, we reindex the sums over $s$ as $\sum_{s=p}^{p+2}$ on the right hand side. Then the bracket evaluates to
\begin{multline} \label{eqn:DtoN.solnlemh.3.1}
	h_p \Big(
		\Msup{B}{p}{p}{r} \Msups{\mathcal{C}}{p+1}{p+2}{p+1}{p+2}(j_0,j_1)
			+ \Msup{B}{p+1}{p}{r} \Msups{\mathcal{C}}{p+1}{p+2}{p+2}{p}(j_0,j_1)
			+ \Msup{B}{p+2}{p}{r} \Msups{\mathcal{C}}{p+1}{p+2}{p}{p+1}(j_0,j_1)
		\Big) \\
	+ h_{p+1} \Big(
		\Msup{B}{p+1}{p}{r} \Msups{\mathcal{C}}{p+2}{p}{p+2}{p}(j_0,j_1)
			+ \Msup{B}{p+2}{p}{r} \Msups{\mathcal{C}}{p+2}{p}{p}{p+1}(j_0,j_1)
			+ \Msup{B}{p}{p}{r} \Msups{\mathcal{C}}{p+2}{p}{p+1}{p+2}(j_0,j_1)
		\Big) \\
	+ h_{p+2} \Big(
		\Msup{B}{p+2}{p}{r} \Msups{\mathcal{C}}{p}{p+1}{p}{p+1}(j_0,j_1)
			+ \Msup{B}{p}{p}{r} \Msups{\mathcal{C}}{p}{p+1}{p+1}{p+2}(j_0,j_1)
			+ \Msup{B}{p+1}{p}{r} \Msups{\mathcal{C}}{p}{p+1}{p+2}{p}(j_0,j_1)
		\Big).
\end{multline}
As each term in expression~\eqref{eqn:DtoN.solnlemh.3.1} is summed over $r,j_0,j_1$ separately, we can cyclically permute the indices $(r,j_0,j_1)\mapsto(j_0,j_1,r)$ in the second term on each line of~\eqref{eqn:DtoN.solnlemh.3.1}, and $(r,j_0,j_1)\mapsto(j_1,r,j_0)$ in the third term on each line of~\eqref{eqn:DtoN.solnlemh.3.1}.
It is then immediate from the definition of $\mathcal{C}$, that the first line of~\eqref{eqn:DtoN.solnlemh.3.1} evalutes to $h_p \mathcal{D}(r,j_0,j_1)$.
However, the second and third lines of~\eqref{eqn:DtoN.solnlemh.3.1} evaluate to
\BES
	h_{p+1} \det
		\BP
			\Msup{B}{p+1}{p}{r} & \Msup{B}{p+1}{p+2}{r} & \Msup{B}{p+1}{p}{r} \\
			\Msup{B}{p+2}{p}{j_0} & \Msup{B}{p+2}{p+2}{j_0} & \Msup{B}{p+2}{p}{j_0} \\
			\Msup{B}{p}{p}{j_1} & \Msup{B}{p}{p+2}{j_1} & \Msup{B}{p}{p}{j_1}
		\EP
	+ h_{p+2} \det
		\BP
			\Msup{B}{p+2}{p}{r} & \Msup{B}{p+2}{p}{r} & \Msup{B}{p+2}{p+1}{r} \\
			\Msup{B}{p}{p}{j_0} & \Msup{B}{p}{p}{j_0} & \Msup{B}{p}{p+1}{j_0} \\
			\Msup{B}{p+1}{p}{j_1} & \Msup{B}{p+1}{p}{j_1} & \Msup{B}{p+1+1}{p}{j_1}
		\EP
	= 0.
\EES
Hence
\BES
	\sum_{r=0}^m \sum_{s=0}^2 \Msup{B}{s}{p}{r} x_s^r = \frac{h_p}{\Delta}\sum_{r,j_0,j_1=0}^m \mathcal{D}(r,j_0,j_1) = h_p.
	\qedhere
\EES
\end{proof}

\subsection{Determinant lemma}

The following lemma is useful in finding the solution of linear system~\eqref{eqn:DtoN.n}.
Although it is not used in the above presented proofs of lemma~\ref{lem:DtoN.soln.2} or lemma~\ref{lem:DtoN.soln.3}, the authors found it helpful for the original derivation of those results.
It is presented here to facilitate the extension of lemma~\ref{lem:DtoN.soln.2} or lemma~\ref{lem:DtoN.soln.3} to arbitrary spatial order $n$.
\begin{lem} \label{lem:detlem}
	Let $n\in\N$, $d\in\N\cup\{0\}$ and $s\in\Z$.
	Define the $d\times d$ matrix $M=M(n,d,s)$ by
	\BES
		\M{M}{j}{k} = \M{\delta}{j-s}{k}-\M{\delta}{j-s}{k-n}
	\EES
	in terms of the Kronecker $\delta$. Then
	\BES
		\det M =
		\begin{cases}
			1 & \mbox{if } s=0, \\
			(-1)^d & \mbox{if } s=n, \\
			(-1)^{ds} & \mbox{if } n|d \mbox{ and } 1\leq s\leq n-1, \\
			0 & \mbox{otherwise.}
		\end{cases}
	\EES
\end{lem}

\begin{proof}
	If $s<0$ ($s>n$) then $M$ is upper (lower) triangular with $0$ along the diagonal, so has determinant $0$.
	If $s=0$ ($s=n$) then $M$ is upper (lower) triangular with $1$ ($-1$) along the diagonal, so has determinant $1$ ($(-1)^d$).
	It remains only to confirm the result holds for $1\leq s\leq n-1$.
	
	Note that, provided $1\leq s\leq n-1$,
	\begin{align*}
		\det M(n,n,s) &= \det\left[\left(\M{\delta}{j-s}{k}-\M{\delta}{j-s}{k-n}\right)_{j,k=1}^n\right] \\
		&= \det\left[\left(\M{\delta}{j}{k}\right)_{j,k=1}^{n-s}-\left(\M{\delta}{j}{k}\right)_{j,k=n-s+1}^n\right](-1)^{(n-1)s},
	\end{align*}
	after $(n-1)s$ row swaps.
	Hence
	\BE \label{eqn:detlem.proof.1}
		\det M(n,n,s) = (-1)^s(-1)^{(n-1)s}=(-1)^{ns}.
	\EE
	
	If $d=0$ then $n|d$ and, by convention, $\det M=1=(-1)^{0s}$.
	Suppose $n>d$ and $1\leq s \leq d$.
	Then row $s$ of $M$ is given by
	\BES
		\left(\M{M}{s}{k}\right)_{k=1}^d = \left( \M{\delta}{0}{k}-\M{\delta}{0}{k-n} \right)_{k=1}^d = (0)_{k=1}^d,
	\EES
	as $0<k$ and $0>d-n\geq k-n$.
	Hence $\det M=0$.
	Now suppose $n>d$ and $d\leq s \leq n-1$.
	Then row $d$ of $M$ is given by
	\BES
		\left(\M{M}{d}{k}\right)_{k=1}^d = \left( \M{\delta}{d-s}{k}-\M{\delta}{d-s}{k-n} \right)_{k=1}^d = (0)_{k=1}^d,
	\EES
	as $d-s\leq0<k$ and $k-n\leq d-n<d-s$. Hence $\det M=0$.
	So the result holds for $1\leq s\leq n-1$ with $0\leq d\leq n-1$.
	Next we provide an inductive step of size $n$ on $d$.
	
	Fix some $n\in\N$, $d\in\N\cup\{0\}$ and $1\leq s \leq n-1$ and suppose
	\BE \label{eqn:detlem.proof.2}
		\det M(n,d,s) = \begin{cases} (-1)^{ds} & \mbox{if } n|d, \\ 0 & \mbox{otherwise.} \end{cases}
	\EE
	Now consider the matrix
	\BES
		M(n,d+n,s) = \BP M(n,n,s) & A \\ B & M(n,d,s) \EP,
	\EES
	where
	\BES
		\M{A}{j}{k}=-\M{\delta}{j-s}{k}, \qquad \M{B}{j}{k}=\M{\delta}{j-s}{k-n}.
	\EES
	As the the first $s$ rows of $A$ are $0$, the first $s$ rows of $M(n,d+n,s)$ have their only nonzero entry in columns $n-s+1, \ldots, n$.
	As the first $n-s$ columns of $B$ are $0$, the first $n-s$ columns of $M(n,d+n,s)$ have their only nonzero entry in rows $s+1, \ldots, n$.
	Hence
	\BES
		\det M(n,d+n,s) = \det M(n,n,s) \det M(n,d,s),
	\EES
	so, by equations~\eqref{eqn:detlem.proof.1} and~\eqref{eqn:detlem.proof.2},
	\BES
		\det M(n,d+n,s) = \begin{cases} (-1)^{(d+n)s} & \mbox{if } n|d, \mbox{ equivalently } n|(d+n), \\ 0 & \mbox{otherwise.} \end{cases}
	\EES
	Hence, by induction, the result holds for all $n$, $d$, and $s$.
\end{proof}

\section{Appendix: Proof of Lemma~\ref{lem:uniq.solve.lem.2}} \label{sec:AppB}

\begin{proof}[Proof of lemma~\ref{lem:uniq.solve.lem.2}]
	We present a full proof of (a).
	The proof of (b) is entirely analogous.
	If $\lambda\in\C^+$, then
	\BE \label{eqn:DominantExp}
		|E_m(\lambda)| > \ldots > |E_0(\lambda)| = 1 = |E_0(-\lambda)| > \ldots > |E_m(-\lambda)|.
	\EE
	Hence, as $\lambda\to\infty$ from within $\C^+_\theta$,
	\BES
		\Delta(\lambda) = \delta^+(\lambda) + \mathcal{O}\left(E_{m-1}(\lambda)\right).
	\EES
	It is immediate from the definition of $\Msup{B}{j}{k}{r}$ that
	\BES
		\delta^+(\lambda) = E_m(\lambda)\left[ k_0+k_1\lambda^{-1}+k_2\lambda^{-2} \right],
	\EES
	for some coefficients $k_0,k_1,k_2\in\C$, of which we have assumed not all are zero.
	Therefore, any term in $x_0(\lambda)\Delta(\lambda)$ which is $\mathcal{O}(E_m(\lambda)\lambda^{-3})$ certainly corresponds to a term of $x_0(\lambda)$ that is $\mathcal{O}(\lambda^{-1})$, which may be ignored.
	The remainder of the proof establishes that the dominant term in $x_0(\lambda)\Delta(\lambda)$ is $\gamma^+(\lambda)$.
	
	Integrating by parts thrice, and applying the Riemann-Lebesgue lemma to control the remainder,
	\begin{subequations} \label{eqn:qT.asymp}
	\begin{align} \notag
		\hat{q}_\tau^r(\lambda) &= \int_{\eta_{r-1}}^{\eta_r}\re^{-i\lambda x}q(x,\tau)\D x = \left[\re^{-i\lambda x}\left(\frac{i}{\lambda}q(x,\tau)+\frac{1}{\lambda^2}q_x(x,\tau)\right)\right]_{x=\eta_{r-1}}^{x=\eta_r} + \mathcal{O}\left(E_r(\lambda)\lambda^{-3}\right) \\ \label{eqn:qT.asymp1}
		&= E_r(\lambda)\left[ \lambda^{-1}iq(\eta_r,\tau) + \lambda^{-2}q_x(\eta_r,\tau) \right] + \mathcal{O}\left(E_r(\lambda)\lambda^{-3}\right), \\ \label{eqn:qT.asymp2}
		\hat{q}_\tau^{r}(-\lambda) &= E_{r-1}(-\lambda)\left[ \lambda^{-1}iq(\eta_{r-1},\tau) - \lambda^{-2}q_x(\eta_{r-1},\tau) \right] + \mathcal{O}\left(E_{r-1}(\lambda)\lambda^{-3}\right).
	\end{align}
	\end{subequations}
	By asymptotic expansions~\eqref{eqn:qT.asymp} and~\eqref{eqn:DominantExp}, all dominant terms must be among those listed below
	\begin{multline} \label{eqn:x0Numerator.1}
		x_0(\lambda)\Delta(\lambda) = \left( \sum_{j=1}^m \left[\Msup{B}{0}{0}{j}\Msup{B}{1}{1}{m}-\Msup{B}{1}{0}{m}\Msup{B}{0}{1}{j}\right](-1)\hat{q}_\tau^j(\lambda) \right. \\
		\left. + \sum_{j,k=0}^m \frac{1}{2} \left[\Msup{B}{1}{0}{j}\Msup{B}{1}{1}{k}-\Msup{B}{1}{0}{k}\Msup{B}{1}{1}{j}\right]
			\left\{\begin{smallmatrix}1\text{\textup{ if }}j<k \\ 0\text{\textup{ if }}j=k \\ -1\text{\textup{ if }}j>k \end{smallmatrix}\right\} \sum_{l=\min\{j,k\}+1}^{\max\{j,k\}} \hat{q}_\tau^l(-\lambda) \right) + \mathcal{O}\left(E_{m-1}(\lambda)\lambda^{-1}\right)
	\end{multline}
	The first term is dominant in the sum over $l$, and the double sum over $j,k$ can be rewritten
	\BES
		\sum_{j=0}^{m-1}\sum_{k=j+1}^m \det\BP\Msup{B}{1}{0}{j}&\Msup{B}{1}{1}{j}\\\Msup{B}{1}{0}{k}&\Msup{B}{1}{1}{k}\EP \hat{q}_\tau^{j+1}(-\lambda),
	\EES
	whose dominant terms occur for $k=m$.
	Expanding the determinant, this sum evaluates to
	\begin{multline}
		\frac{1}{4}\sum_{j=0}^{m-1}E_j(\lambda)E_m(\lambda)\det
			\BP
				-\Msup{b}{0}{0}{j}\frac{1}{i\lambda}+\Msup{b}{1}{0}{j} & -\Msup{b}{0}{1}{j}\frac{1}{i\lambda}+\Msup{b}{1}{1}{j} \\
				B_0 & B_1
			\EP
			E_j(-\lambda)\left[ \lambda^{-1}iq(\eta_j,\tau) - \lambda^{-2}q_x(\eta_j,\tau) \right] \\
			+ \mathcal{O}\left(E_m(\lambda)\lambda^{-3}\right),
	\end{multline}
	where
	\BES
		B_0 := -\Msup{b}{0}{0}{m}\frac{1}{i\lambda}+\Msup{b}{1}{0}{m}
		\qquad\qquad\mbox{and}\qquad\qquad
		B_1 := -\Msup{b}{0}{1}{m}\frac{1}{i\lambda}+\Msup{b}{1}{1}{m}
	\EES
	are defined for notational convenience.
	Similarly, the first sum in equation~\eqref{eqn:x0Numerator.1} can be expanded to
	\BES
		\frac{-1}{4} \sum_{j=1}^m E_j(-\lambda)E_m(\lambda) \det
			\BP
				\Msup{b}{0}{0}{j}\frac{1}{i\lambda}+\Msup{b}{1}{0}{j} & \Msup{b}{0}{1}{j}\frac{1}{i\lambda}+\Msup{b}{1}{1}{j} \\
				B_0 & B_1
			\EP
			E_j(\lambda)\left[ \lambda^{-1}iq(\eta_j,\tau) + \lambda^{-2}q_x(\eta_j,\tau) \right] + \mathcal{O}\left(E_m(\lambda)\lambda^{-3}\right).
	\EES
	Exploiting linearity of the determinants in their top rows, and after some judicious cancellation, we arrive at
	\begin{multline} \label{eqn:x0Numerator.2}
		x_0(\lambda)\Delta(\lambda) = \frac{-1}{2}E_m(\lambda)
			\lambda^{-2} \sum_{j=1}^{m-1} \left(
				\det \BP \Msup{b}{0}{0}{j} & \Msup{b}{0}{1}{j} \\ B_0 & B_1 \EP q(\eta_j,\tau)
				+
				\det \BP \Msup{b}{1}{0}{j} & \Msup{b}{1}{1}{j} \\ B_0 & B_1 \EP q_x(\eta_j,\tau)
			\right)
		\\
		+ \frac{1}{4}E_m(\lambda) \det
			\BP
				-\Msup{b}{0}{0}{0}\frac{1}{i\lambda}+\Msup{b}{1}{0}{0} & -\Msup{b}{0}{1}{0}\frac{1}{i\lambda}+\Msup{b}{1}{1}{0} \\
				B_0 & B_1
			\EP
			\left[ \lambda^{-1}iq(\eta_j,\tau) - \lambda^{-2}q_x(\eta_j,\tau) \right]
		\\
		- \frac{1}{4}E_m(\lambda) \det
			\BP
				\Msup{b}{0}{0}{m}\frac{1}{i\lambda}+\Msup{b}{1}{0}{m} & \Msup{b}{0}{1}{m}\frac{1}{i\lambda}+\Msup{b}{1}{1}{m} \\
				B_0 & B_1
			\EP
			\left[ \lambda^{-1}iq(\eta_j,\tau) + \lambda^{-2}q_x(\eta_j,\tau) \right]
		+ \mathcal{O}\left(E_m(\lambda)\lambda^{-3}\right).
	\end{multline}
	
	Discarding $\mathcal{O}\left(E_m(\lambda)\lambda^{-3}\right)$ terms, the lower rows of the determinants in the first line of equation~\eqref{eqn:x0Numerator.2} may be replaced by $(\Msup{b}{1}{0}{m},\Msup{b}{1}{1}{m})$.
	As
	\BES
		\frac{1}{4}E_m(\lambda)\det
			\BP
				-\Msup{b}{0}{0}{0}\frac{1}{i\lambda}+\Msup{b}{1}{0}{0} & -\Msup{b}{0}{1}{0}\frac{1}{i\lambda}+\Msup{b}{1}{1}{0} \\
				B_0 & B_1
			\EP
		= \delta^+(\lambda) - \frac{1}{2i\lambda} E_m(\lambda) \det \BP \Msup{b}{0}{0}{0} & \Msup{b}{0}{1}{0} \\ \Msup{b}{1}{0}{m} & \Msup{b}{1}{1}{m} \EP + \mathcal{O}\left(E_m(\lambda)\lambda^{-2}\right),
	\EES
	the second line of equation~\eqref{eqn:x0Numerator.2} can be replaced by
	\BES
		- \frac{1}{2}E_m(\lambda) \lambda^{-2} \det \BP \Msup{b}{0}{0}{0} & \Msup{b}{0}{1}{0} \\ \Msup{b}{1}{0}{m} & \Msup{b}{1}{1}{m} \EP q(\eta_j,\tau) + \mathcal{O}\left(\delta^+(\lambda)\lambda^{-1}\right).
	\EES
	Moreover, $E_m(\lambda)\lambda^{-3}=\mathcal{O}(\delta^+(\lambda)\lambda^{-1})$, so the new error term $\mathcal{O}(\delta^+(\lambda)\lambda^{-1})$ includes the previous error term $\mathcal{O}(E_m(\lambda)\lambda^{-3})$.
	The matrix whose determinant is calculated in the third line of equation~\eqref{eqn:x0Numerator.2} has linearly dependent $\mathcal{O}(1)$ and $\mathcal{O}(\lambda^{-1})$ components of its entries, so it may be replaced by a simpler determinant, which is a constant multiple of $\lambda^{-1}$.
	So the second and third lines of equation~\eqref{eqn:x0Numerator.2} simplify to the additional terms $j=0$ and $j=m$ in the first sum on the first line of equation~\eqref{eqn:x0Numerator.2} (plus lower order terms).
	That is
	\begin{multline} \label{eqn:x0Numerator.3}
		x_0(\lambda)\Delta(\lambda) = \frac{-1}{2}E_m(\lambda) \lambda^{-2}
			\left[
				\sum_{j=0}^{m} \det \BP \Msup{b}{0}{0}{j} & \Msup{b}{0}{1}{j} \\ \Msup{b}{1}{0}{m} & \Msup{b}{1}{1}{m} \EP q(\eta_j,\tau)
				+
				\sum_{j=1}^{m-1} \det \BP \Msup{b}{1}{0}{j} & \Msup{b}{1}{1}{j} \\ \Msup{b}{1}{0}{m} & \Msup{b}{1}{1}{m} \EP q_x(\eta_j,\tau)
			\right] \\
		+ \mathcal{O}\left(\delta^+(\lambda)\lambda^{-1}\right).
	\end{multline}
	Note that
	\BES
		\det \BP \Msup{b}{1}{0}{m} & \Msup{b}{1}{1}{m} \\ \Msup{b}{1}{0}{m} & \Msup{b}{1}{1}{m} \EP = 0,
	\EES
	so the upper limit of the latter sum in equation~\eqref{eqn:x0Numerator.3} may be increased to $m$. Further,
	\BES
		E_m(\lambda) \det \BP \Msup{b}{1}{0}{0} & \Msup{b}{1}{1}{0} \\ \Msup{b}{1}{0}{m} & \Msup{b}{1}{1}{m} \EP = 4\delta^+(\lambda) + \mathcal{O}\left(E_m(\lambda)\lambda^{-1}\right),
	\EES
	so the lower limit of the latter sum in equation~\eqref{eqn:x0Numerator.3} may be reduced to $0$, and
	\BES
		x_0(\lambda)\Delta(\lambda) = \frac{-1}{2\lambda^2}E_m(\lambda) \sum_{j=0}^{m}\left[ \det\BP\Msup{b}{0}{0}{j}&\Msup{b}{0}{1}{j}\\\Msup{b}{1}{0}{m}&\Msup{b}{1}{1}{m}\EP q(\eta_j,\tau) + \det\BP\Msup{b}{1}{0}{j}&\Msup{b}{1}{1}{j}\\\Msup{b}{1}{0}{m}&\Msup{b}{1}{1}{m}\EP q_x(\eta_j,\tau) \right] + \mathcal{O}\left(\delta^+(\lambda)\lambda^{-1}\right).
	\EES
	Finally, exploiting linearity of the determinants in their first rows, and applying the multipoint conditions~\eqref{eqn:IMVP:MC},
	\BES
		x_0(\lambda)\Delta(\lambda) = \gamma^+(\lambda) + \mathcal{O}\left(\delta^+(\lambda)\lambda^{-1}\right). \qedhere
	\EES
\end{proof}

\bibliographystyle{amsplain}
\bibliography{dbrefs}

\end{document}